\documentclass[a4paper]{article}
\usepackage{amsfonts,amsthm,amsmath,amssymb,mathtools}
\usepackage{array}
\oddsidemargin=-\evensidemargin
\newtheorem{lemma}{Lemma}
\newtheorem{proposition}{Proposition}
\newtheorem{theorem}{Theorem}
\newtheorem{example}{Example}
\theoremstyle{definition}
\newtheorem{definition}{Definition}

\newcommand{\Ga}{\mathbf G_a}
\newcommand{\Aut}{\mathop{\mathrm{Aut}}}
\newcommand{\Bih}{\mathop{\mathrm{Bih}}}
\newcommand{\Lie}{\mathop{\mathrm{Lie}}}
\newcommand{\ad}{\mathop{\mathrm{ad}}}
\newcommand{\id}{\mathrm{id}}
\newcommand{\tr}{\mathop{\mathrm{tr}}}
\newcommand{\rk}{\mathop{\mathrm{rk}}}
\newcommand{\Ad}{\mathop{\mathrm{Ad}}}
\newcommand{\PP}{\mathbf{P}}
\newcommand{\veps}{\varepsilon}

\begin{document}
\title{Unipotent commutative group actions on flag varieties and nilpotent multiplications}
\author{Rostislav Devyatov\footnote{Supported in part by the Simons Foundation}\\Math Department
of Higher School of Economics, Moscow\\\texttt{deviatov@mccme.ru}}
\maketitle

\begin{abstract}
Our goal is to classify all generically transitive actions of commutative unipotent groups on flag varieties 
up to conjugation. We establish relationship between this problem and classification 
of multiplications with certain properties on Lie algebra representations. Then we classify 
multiplications with the desired properties and solve the initial classification problem.
\end{abstract}

\section{Introduction}
Let $G$ be a semisimple algebraic group over $\mathbb C$.
Consider a \textit{generalized flag variety} $G/P$ where $P\subset G$ is a parabolic subgroup.
Let $m=\dim G/P$, and let $(\Ga)^m$ be the unipotent commutative group 
of dimension $m$.
We are going to classify all generically transitive actions of $(\Ga)^m$ on $G/P$ up to conjugation by 
an automorphism of $G/P$. The variety $G/P$ does not change after taking the quotient 
of $G$ and of $P$ over a finite central subgroup simultaneously, so in the sequel 
we suppose that the center of $G$ is trivial.

If $G/P$ is a projective space, the actions in 
question were classified in \cite{htch}. In \cite{arj}, all possible pairs $(G, P)$ such that at least one 
commutative unipotent action is possible were found, and the problem of classification 
of actions on Grassmannians was stated.

If $G=G^{(1)}\times\ldots\times G^{(s)}$ is the factorization of $G$ into a product of 
simple subgroups, and $P^{(i)}=G^{(i)}\cap P$, then (see \cite[Chapter 4, \S15.4, Theorem 2]{onisch} 
and Section \ref{autgroups} here) 
the group $\widetilde G=\Aut(G/P)^\circ$ 
can be written as $\widetilde G=\widetilde G^{(1)}\times\ldots\times \widetilde G^{(s)}$, 
where $\widetilde G^{(i)}=\Aut(G^{(i)}/P^{(i)})^\circ$, and the action is diagonal. 
Here by the automorphism group 
of an algebraic variety we understand the following.
\begin{definition}\label{myautdef}
Let $X$ be an algebraic variety. An algebraic group $G$ together with an 
algebraic action $G:X$ is called 
the \textit{automorphism group of $X$} if for every algebraic group $H$ acting 
on $X$ algebraically there exists a unique algebraic group morphism $f\colon G\to H$ 
such that for every $x\in X$ and $h\in H$ one has $h\cdot x=f(h)\cdot x$.
\end{definition}
This enables us to consider subgroups 
of $\Aut (G/P)^\circ$ isomorphic to $(\Ga)^m$ instead of actions $(\Ga)^m:(G/P)$ and reduces 
the problem to the case when $G$ is simple.
The existence of the automorphism group for generalized flag varieties 
in this sense
will be established 
in Section \ref{autgroups}.  Note that in \cite{onisch}
a different definition of the automorphism group is used.

Moreover, it follows from the same theorem in \cite{onisch} and Section \ref{autgroups} 
here that if $G$ is simple, then $\widetilde G=\Aut (G/P)^\circ$ 
is also simple with trivial center, and there exists 
a parabolic subgroup $\widetilde P\subseteq \widetilde G$ such that $G/P$ is 
$\widetilde G$-isomorphic to $\widetilde G/\widetilde P$, where the action 
on $\widetilde G/\widetilde P$ originates from the left action $\widetilde G:\widetilde G$.
More precisely, for $(G,P)=(PSp_{2l}, P_1)$, 
$(\text{group of type $G_2$}, P_1)$, 
$(SO_{2l-1}, P_{l-1})$,  
one has $(\widetilde G, \widetilde P)=(PSL_{2l},P_1)$, 
$(SO_7,P_1)$ or $(PSO_{2l},P_l)$, respectively, and for all other pairs $(G\text{ simple}, P\text{ parabolic})$
one has $\widetilde G=G$, $\widetilde P=P$. 
The classification problem 
is now reduced to
the cases where $G$ is simple and $G=\Aut (G/P)^\circ$.
(Here and further $P_i$ denotes the maximal parabolic subgroup corresponding to the $i$-th simple root, 
roots are enumerated as in \cite{bou}.)


All pairs of a simple group $G$ and its parabolic subgroup $P$ up to isogeny
such that $G=\Aut (G/P)^\circ$ and $G/P$ allows a generically transitive $(\Ga)^m$-action are 
listed in the following table, see \cite[Theorem 1]{arj}. 
$$
\begin{array}{c|c}
G & P \\
\hline\hline
PSL_{l+1} & P_i\text{ ($1\le i\le l$)} \\
\hline
SO_{2l+1} & P_1 \\
\hline
PSp_{2l} & P_l \\
\hline
PSO_{2l} & P_i\text { ($i=1,l-1,l$)} \\
\hline
\text{Group of type } E_6 & P_i\text{ ($i=1,6$)} \\
\hline
\text{Group of type } E_7 & P_7 \\
\end{array}
$$
Note that \cite[Theorem 1]{arj} in all these cases the unipotent radical of $P$ is commutative. Hence, 
to classify generically transitive $(\Ga)^m$-actions on generalized flag varieties, it is sufficient 
to consider only varieties of the form $G/P$, where $G$ is simple, 
and $P$ is its parabolic subgroup such 
that $G=\Aut (G/P)^\circ$ and the unipotent radical of $P$ is commutative.

Every Lie algebra in what follows is a subalgebra of the Lie algebra of a reductive algebraic group $H$, and 
it will be considered together with this embedding. We call such a Lie algebra $\mathfrak a$ \textit{unipotent}, 
if it is the Lie algebra of a unipotent algebraic subgroup of $H$. In other words, $\mathfrak a$ is 
called unipotent if it is a subalgebra of the Lie algebra of a maximal unipotent subgroup of $H$. Alternatively, 
if $V$ is a representation of $H$ with a finite kernel, $\mathfrak a$ is unipotent if every element of $\mathfrak a$ 
acts on $V$ by a nilpotent operator.

\begin{proposition}\label{gaactsubalg}
(see Section \ref{reduction}) 
Let $G$ be a simple algebraic group, $P$ be a parabolic subgroup such 
that $\Aut (G/P)^\circ=G$ and the unipotent radical of $P$ is commutative, $\mathfrak g=\Lie G$, 
$\mathfrak p=\Lie P$. Then there is a bijection between generically transitive actions $(\Ga)^m:(G/P)$
up to $G$-conjugation and commutative unipotent subalgebras $\mathfrak a\subset\mathfrak g$ such that 
$\mathfrak a\cap\mathfrak p=0$ and $\mathfrak a\oplus\mathfrak p=\mathfrak g$ up to $P$-conjugation.
\end{proposition}

Fix a Borel subgroup $B\subset G$ and a maximal torus $T\subset B$. Let $B^-\subset G$ be the Borel subgroup 
such that $B\cap B^-=T$. Then for every parabolic subgroup $P\subseteq G$ such that $B\subseteq P$ 
there exists a unique 
parabolic subgroup $P^-$ such that $B^-\subseteq P^-$ and $L=P^-\cap P$ is a Levi subgroup of $P$.
Let $U^-$ be the unipotent radical of $P^-$, $\mathfrak u^-=\Lie U^-$. Then $\mathfrak g=\mathfrak u^-\oplus\mathfrak p$.
We also have a decomposition $\mathfrak g=\mathfrak a\oplus\mathfrak p$, so every $u\in \mathfrak u^-$ can be written as 
$u=a+p$ for some $p\in \mathfrak p$. Set $\varphi (u)=-p$. Clearly, $\varphi\colon\mathfrak u^-\to\mathfrak p$ is 
a linear map and for every $u\in\mathfrak u^-$ we have $u+\varphi(u)\in\mathfrak a$. 
On the other hand, every $a\in\mathfrak a$ can be written as $a=v+q$, $v\in\mathfrak u^-$, 
$q\in \mathfrak p$, and we see from the definition of $\varphi$ that $q=\varphi (v)$ satisfies this equation. 
Hence, $\mathfrak a=\{u+\varphi (u)\mid u\in\mathfrak u^-\}$. 
Note that this correspondence between subalgebras $\mathfrak a$ and maps $\varphi$ is 
compatible with the $P$-actions on the set of $m$-dimensional subalgebras of $\mathfrak g$, 
on $\mathfrak u^-$, and on $\mathfrak p$.
For a general subalgebra $\mathfrak a$ from 
Proposition \ref{gaactsubalg} we can only say that $\varphi (\mathfrak u^-)\subseteq\mathfrak p$, but 
the following proposition shows that we can say more about $\varphi (\mathfrak u^-)$ if we apply 
a suitable conjugation by an element of $P$ to $\mathfrak a$. Denote the unipotent radical of 
$B^-$ by $U_0^-$.

\begin{proposition}\label{subalgisunip}
(see Section \ref{reduction}) 
For every unipotent commutative subalgebra $\mathfrak a\subset \mathfrak g$ such that $\mathfrak a\cap\mathfrak p=0$
and $\mathfrak a\oplus\mathfrak p=\mathfrak g$ there exists 
$p\in P$ such that $(\Ad p)\mathfrak a\subseteq \mathfrak u_0^-=
\Lie U_0^-$.
\end{proposition}

If $\mathfrak a\subseteq \mathfrak u_0^-$, then for every 
$u\in\mathfrak u^-$ we have $u+\varphi (u)\in\mathfrak u_0^-$, so,
since $\mathfrak u^-$ and $\mathfrak u_0^-$ are subspaces of 
$\mathfrak p^-$,
$\varphi(u)\in\mathfrak p^-\cap\mathfrak p=\mathfrak l=\Lie L$.
Therefore, every $P$-conjugation class of 
commutative unipotent subalgebras $\mathfrak a\subset\mathfrak g$ such that $\mathfrak a\cap\mathfrak p=0$
and $\mathfrak a\oplus\mathfrak p=\mathfrak g$ can be defined by some 
$L$-conjugation class of 
linear maps $\varphi\colon \mathfrak u^-\to\mathfrak l$. 
Notice that it is not true in general that every linear map $\varphi\colon \mathfrak u^-\to\mathfrak l$
leads to a suitable subalgebra.
It could also turn out that several classes of maps up to $L$-conjugation
define the same class of subalgebras up to $P$-conjugation, but later we will see 
that this is a bijection.

Such a map $\varphi$ enables us to define a multiplication $\mathfrak u^-\times\mathfrak u^-\to\mathfrak u^-$, 
namely $u\times v=[\varphi(u),v]$. 
Since the representation $\mathfrak l:\mathfrak u^-$ is faithful, 
given such a multiplication defined by a linear map $\mathfrak u^-\to\mathfrak l$, it is possible
to recover the linear map. The following theorem describes the multiplications 
$\mathfrak u^-\times\mathfrak u^-\to\mathfrak u^-$ that are really defined my maps 
$\varphi\colon \mathfrak u^-\to\mathfrak l$ such that 
$\mathfrak a=\{u+\varphi (u)\mid u\in\mathfrak u^-\}$ is a commutative unipotent subalgebra.
Here and further, if a vector space $V$ is equipped with 
a multiplication $V\otimes V\to V$, and $v\in V$, we denote by $\mu_v$ the operator 
$\mu_v\colon V\to V$ 
defined by $\mu_v w=vw$. We shortly call it a \textit{multiplication operator}.

\begin{theorem}\label{actsandmults}
(see Section \ref{reduction}) 
A multiplication $\mathfrak u^-\times\mathfrak u^-\to\mathfrak u^-$ defines 
a commutative unipotent subalgebra 
$\mathfrak a\subset\mathfrak p^-$ 
satisfying 
$\mathfrak a\cap\mathfrak p=0$
and $\mathfrak a\oplus\mathfrak p=\mathfrak g$
as described above if and only if this multiplication is 
commutative, associative, and every multiplication operator 
$\mu_w\colon\mathfrak u^-\to \mathfrak u^-$ ($w\in \mathfrak u^-$)
is nilpotent and 
coincides with an operator of the form $\ad g|_{\mathfrak u^-}$, where $g\in \mathfrak l$.
\end{theorem}

Thus, to classify generically transitive $(\Ga)^m$-actions on $G/P$, it is sufficient 
to find all multiplications satisfying the conditions from Theorem \ref{actsandmults} 
on certain representations of certain reductive Lie algebras. This problem can be naturally generalized 
as follows.
Let $L$ be a connected reductive algebraic group, let $V$ 
be a representation of $L$. These data define an action of $\mathfrak l=\Lie L$ on $V$. 
Denote the action of an element $x\in\mathfrak l$ 
on $V$ by $\rho(x)$.
We want to classify up to $L$-conjugation all multiplications $V\times V\to V$ satisfying the 
conditions from Theorem \ref{actsandmults}, namely:
\begin{enumerate}
\item\label{firstcompcond} The multiplication is commutative.
\item The multiplication is associative.
\item Every multiplication operator $\mu_w\colon V\to V$ 
($w\in V$)
is nilpotent.
\item\label{lastcompcond} For every $w\in W$ there exists $x\in \mathfrak l$ such that $\rho(x)=\mu_w$.
\end{enumerate}
We call a multiplication $V\times V\to V$ \textit{$\mathfrak l$-compatible} if conditions
\ref{firstcompcond}--\ref{lastcompcond} hold for it.

First, let us reduce this problem to the case of an irreducible representation of a simple group.
Since all multiplication operators $\mu_v$ are nilpotent, the elements $x\in \mathfrak l$ such that $\rho (x)=\mu_v$
can be taken from $[\mathfrak l, \mathfrak l]$, so we will always suppose that $x\in[\mathfrak l, \mathfrak l]$.
Moreover, we are going to prove the following proposition:
\begin{proposition}\label{possibledecomp} (see Section \ref{multgeneral})
Let $L$ be a connected reductive group, $\mathfrak l=\Lie L$. 
Let $V$ be a representation of $L$ such that there exists a nonzero 
$\mathfrak l$-compatible multiplication
$V\times V\to V$.
Let $[\mathfrak l, \mathfrak l]=\mathfrak l_1\oplus\ldots\oplus\mathfrak l_s$ 
be the decomposition of $[\mathfrak l, \mathfrak l]$ into simple summands.
Then, after a suitable permutation of the subalgebras $\mathfrak l_i$,
there exists a decomposition $V=V_1\oplus\ldots\oplus V_t$ 
and an index $r\le s$, $r\le t$ such that:
\begin{enumerate}
\item $V_i$ is an irreducible representation of $\mathfrak l_i$ for $1\le i\le r$.
\item $\mathfrak l_i\cdot V_j=0$ for $1\le i, j\le r$, $i\ne j$.
\item $V_iV_j=0$ for $1\le i, j\le r$, $i\ne j$.
\item $\mathfrak l_i\cdot V_j=0$ for $1\le i\le r$, $r<j\le t$.
\item $\mathfrak l_i\cdot V_j=0$ for $r<i\le s$, $1\le j\le r$.
\item $V_iV=0$ for $r<i\le t$.
\end{enumerate}
\end{proposition}

Informally speaking, the representation $[\mathfrak l,\mathfrak l]:V$ can be decomposed into a sum of two 
direct summands, "the nontrivial summand" $\mathfrak l_1\oplus\ldots\oplus\mathfrak l_r:V_1\oplus\ldots\oplus V_r$
and "the trivial summand" $\mathfrak l_{r+1}\oplus\ldots\oplus\mathfrak l_s:V_{r+1}\oplus\ldots\oplus V_t$.
The "trivial" part of the algebra or the "trivial" part of the representation (or both) can be zero.
The action of the "trivial" part of $[\mathfrak l,\mathfrak l]$ on the "nontrivial" part of $V$ is zero, as 
well as the action of the "nontrivial" part of $[\mathfrak l,\mathfrak l]$ on the "trivial" part of $V$.
The multiplication between the "nontrivial" and the "trivial" part of $V$ is also zero, 
as well as the multiplication inside the "trivial" part of $V$. The "nontrivial" parts of $[\mathfrak l,\mathfrak l]$
and $V$ can be further decomposed into direct sums of simple algebras and 
irreducible representations of \textbf{each of the simple algebras}. The multiplication between 
different irreducible subrepresentations inside the "nontrivial" part of $V$ is also zero.

Notice that toric direct summands in $\mathfrak l$ and simple direct summands $\mathfrak l_{r+1}, \ldots, \mathfrak l_s$
play different roles. The direct summands $\mathfrak l_{r+1}, \ldots, \mathfrak l_s$ must act on 
$V_1\oplus\ldots\oplus V_r$ trivially, otherwise there exist no nonzero $\mathfrak l$-compatible multiplications. And 
the toric direct summands are allowed to act nontrivially on $V_1\oplus\ldots\oplus V_r$, 
this does not change the set of nonzero $\mathfrak l$-compatible multiplications.
Generally speaking, it is possible that two $\mathfrak l$-compatible multiplications 
are in the same class up to $L$-conjugation, but not in the same class up to 
conjugation by the semisimple part of $L$.

The central torus $T$ of $L$ acts on the space of all multiplications $V_i\times V_i\to V_i$ ($1\le i\le r$) by 
a character $\chi_i$, i.~e. if $g\in T$, 
the action of $g$ multiplies all structure constants of a multiplication $V_i\times V_i\to V_i$
by $\chi_i(g)$. For the action of the simple subgroup $L_i\subseteq L$ ($1\le i\le r$) such that $\Lie L_i=\mathfrak l_i$, 
there are two possibilities:
\begin{enumerate}
\item\label{goodcase} For every $c_i\in\mathbb C$ and for every $\mathfrak l_i$-compatible multiplication $V_i\times V_i\to V_i$
there exists $g\in L_i$ that multiplies all structure constants by $c_i$.
\item There exists an $\mathfrak l_i$-compatible multiplication such that for only finitely many $c_i\in\mathbb C$ 
there exists $g\in L_i$ that multiplies all structure constants by $c_i$.
\end{enumerate}
Later we will see that case \ref{goodcase} always takes place if $\mathfrak l_i$ is not of type $A_l$.
If case \ref{goodcase} holds for all indices $i$, $1\le i\le r$ such that $T$ acts on $V_i$ nontrivially 
and there exist nonzero $\mathfrak l_i$-compatible 
multiplications on $V_i$, then the classes of multiplications up to
conjugation by the semisimple part of $L$ coincide with the classes up to $L$-conjugation.
Otherwise one should 
take the quotient over the action of $T$ explicitly.
We will say more about that in case of the tautological representation of a group of type $A_l$ in 
Subsection \ref{commalgarises}.

Now we are going to consider irreducible representations of simple groups.
\begin{proposition}\label{shortroot} (see Section \ref{multgeneral})
Let $\mathfrak l$ be a simple Lie algebra, denote its root system by $\Phi$. 
Let $V$ be an irreducible representation of $\mathfrak l$ with 
a highest weight $\lambda$ such that there exist nonzero $\mathfrak l$-compatible multiplications on $V$.

Then $\lambda$ is a fundamental weight. Moreover, suppose that $\lambda$ corresponds to the simple 
root $\alpha_i$. Denote the corresponding simple root in the dual root system $\Phi^\vee$ by $\alpha_i^\vee$. 
Then $\alpha_i^\vee$ occurs in the decomposition of the highest short root of $\Phi^\vee$ into a sum of simple root 
only once (i.~e. with coefficient 1).
\end{proposition}

This proposition radically restricts the set of pairs $(\mathfrak l, V)$ where nonzero 
$\mathfrak l$-compatible multiplications are possible. Namely, if $\mathfrak l$ is of type $A_l$, 
we have to consider all fundamental representations, if $\mathfrak l$ is of type $B_l$, 
we have to consider the tautological and the spinor representations, if $\mathfrak l$ is 
of type $C_l$, we have to consider all fundamental representations, if $\mathfrak l$ is 
of type $D_l$, we have to consider the tautological representation and two 
half-spinor representations 
(one of them is transformed to the other one by a diagram automorphism of $\mathfrak l$),
if $\mathfrak l$ is of type $E_6$, $E_7$, $F_4$ or $G_2$, we have to consider only the 
representations of minimal dimension. We are going to prove the following theorem:
\begin{theorem}\label{multsclass} (see Section \ref{multcasebycase})
Let $\mathfrak l$ be a simple Lie algebra, $V$ be an irreducible representation of $\mathfrak l$ 
such that there exists a nonzero $\mathfrak l$-compatible multiplication on $V$. Then there are 
exactly two possibilities:
\begin{enumerate}
\item $\mathfrak l$ is of type $A_l$, $V$ is the tautological representation or the dual one. 
Then an $\mathfrak l$-compatible multiplication is any commutative associative multiplication 
such that all multiplication operators are nilpotent.
\item $\mathfrak l$ is of type $C_l$ ($l\ge 2$), $V$ is the tautological representation.
Then multiplications on $V$ are parametrized by symmetric trilinear forms on $V/V_1$, where $V_1\subset V$ 
is a prefixed Lagrangian subspace. See Subsection \ref{spnontrivmult} for an exact description 
of this parametrization.
\end{enumerate}
\end{theorem}

It should be underlined that if an $\mathfrak l$-compatible multiplication exists on an $\mathfrak l$-module
$V$, then the pair $(\mathfrak l, V)$ is listed only once in this theorem, while $\mathfrak l$ can be isomorphic
to several classical Lie algebras (belonging to several different series). More specifically, here is the list of the cases 
where nontrivial $\mathfrak l$-compatible multiplications exist, but they are not listed directly in Theorem
\ref{multsclass}:

\begin{enumerate}
\item $\mathfrak{so}_5$ is an algebra of type $B_2$, and it isomorphic to $\mathfrak{sp}_4$, 
which is an algebra of type $C_2$. This isomorphism identifies the spinor representation of $\mathfrak{so}_5$
with the tautological representation of $\mathfrak{sp}_4$, so nontrivial $\mathfrak{so}_5$-compatible 
multiplications exist on the spinor representation and can be classified using this identification and Theorem \ref{multsclass}.
\item $\mathfrak{so}_6$ is sometimes referred to as an algebra of type $D_3$, and it is isomorphic to 
$\mathfrak{sl}_4$. This isomorphism identifies the two half-spinor representations of $\mathfrak{so}_6$ 
with the tautological representation of $\mathfrak{sl}_4$ and the dual one. Therefore, nontrivial 
$\mathfrak{so}_6$-compatible multiplications exist on its half-spinor representations and are described 
in the same way as $\mathfrak{sl}_4$-compatible multiplications on the tautological representation and on the dual one.
\item $\mathfrak{sp}_2$ is isomorphic to $\mathfrak{sl}_2$, and the tautological representation of 
$\mathfrak{sp}_2$ is isomorphic to the tautological representation of $\mathfrak{sl}_2$. One can use any 
part of Theorem \ref{multsclass} to describe $\mathfrak l$-compatible multiplications.
\item $\mathfrak{so}_3$ is isomorphic to $\mathfrak{sl}_2$, and the spinor representation of $\mathfrak{so}_3$
is isomorphic to the tautological representation of $\mathfrak{sl}_2$. $\mathfrak{so}_3$-compatible
multiplications on the spinor representation are described as $\mathfrak{sl}_2$-compatible multiplications
on the tautological representation of $\mathfrak{sl}_2$.
\item $\mathfrak{so}_4$ is isomorphic to $\mathfrak{sl}_2\oplus\mathfrak{sl}_2$, so it is not a simple algebra. 
However, each half-spinor representation of $\mathfrak{so}_4$ is isomorphic to the tensor product 
of the tautological representations of one of the $\mathfrak{sl}_2$ algebras and the trivial representation of the other one.
Therefore, nontrivial $\mathfrak{so}_4$-compatible multiplications exist and are described as 
$\mathfrak{sl}_2\oplus\mathfrak{sl}_2$-compatible multiplications on the corresponding tensor product.
\end{enumerate}

For generically transitive $(\Ga)^m$-actions on $G/P$ we then obtain the following theorem:
\begin{theorem}\label{actsclass} (see Section \ref{finalclass})
Let $G$ be a simple group, $P\subset G$ be a parabolic subgroup such that $(\Aut (G/P))^\circ=G$, 
$m=\dim (G/P)$. 

Then if $G$ is of type $A_l$ and $P=P_1$ or $P=P_l$, then generically 
transitive actions $(\Ga)^m:(G/P)$ are parametrized by commutative associative 
$m$-dimensional algebras with nilpotent multiplication operators. Otherwise, 
either there is exactly one generically transitive action $(\Ga)^m:(G/P)$ 
up to $G$-conjugation, this happens if and only if
the unipotent radical of $P$ is commutative, or there are no 
generically transitive actions $(\Ga)^m:(G/P)$.
\end{theorem}

\subsection{Acknowledgments}
I thank prof. Ivan Arzhantzev for bringing my attention to this problem and prof. Ernest Vinberg 
for a useful discussion that helped me to simplify the proofs. I also thank Valentina Kiritchenko, 
Sergey Loktev and Lev Sukhanov for a useful discussion on automorphism groups of flag varieties.

\section{Preliminaries}\label{prelim}
If an algebraic group is denoted by a single capital letter, the corresponding small German letter 
denotes its Lie algebra.
We denote the identity element of a group $G$ by $1_G$.

Starting from Section \ref{reduction} we fix a simple algebraic group $G$, a Borel subgroup $B\subset G$ and 
a maximal torus $T\subset B$. These data determine a root system $\Phi\subset \mathfrak t^*$, 
a subset of positive roots $\Phi^+$ and a subset of simple roots $\Delta=\{\alpha_1,\ldots,\alpha_{\rk\mathfrak g}\}$. Denote 
$\Phi^-=\Phi\setminus\Phi^+$. Denote the root subspace of $\mathfrak g$ corresponding to 
a root $\alpha$ by $\mathfrak g_\alpha$. If $\alpha\in\mathfrak t^*$ is not an element of $\Phi$, 
then we denote by $\mathfrak g_\alpha$ the zero subspace of $\mathfrak g$. For every root $\alpha$ choose elements 
$x_\alpha\in\mathfrak g_\alpha$, $y_\alpha\in\mathfrak g_{-\alpha}$ so that 
together with $h_\alpha=[x_\alpha,y_\alpha]$ they form a standard basis of $\mathfrak{sl}_2$.
Suppose that $x_{-\alpha}=y_\alpha$. Unless stated otherwise, if $\alpha$ is a simple root, $\alpha=\alpha_i$, we shortly
denote $x_{\alpha_i}=x_i$, $y_{\alpha_i}=y_i$, and $h_{\alpha_i}=h_i$.

Parabolic subgroups $P$ containing $B$ are parametrized by subsets $I\subseteq \Delta$, 
namely, a subset $I\subseteq \Delta$ corresponds to the parabolic subgroup $P$ such that
$$
\mathfrak p=\mathfrak b\oplus\bigoplus_{\alpha\in\Phi_I}\mathfrak g_\alpha,
$$
where $\Phi_I\subseteq \Phi^-$ denotes the set of the negative roots
whose decomposition into a sum of simple roots \textit{does not} contain
the roots $\alpha_i$, $i\in I$.
In particular, maximal parabolic subgroups correspond to one-element subsets $I$. Denote 
the subgroup corresponding to $\{\alpha_i\}$ by $P_i$. Every parabolic subgroup is 
conjugate to a parabolic subgroup containing $B$.

Given a parabolic subgroup $P$ containing $B$ (in particular, $P=B$), denote by $P^-$
the parabolic subgroup such that
$$
\mathfrak p^-=\Lie P^-=\mathfrak t\oplus\bigoplus_{\alpha:\mathfrak g_\alpha\subset\mathfrak p}
\mathfrak g_{-\alpha}.
$$
Then $L=P\cap P^-$ is a Levi subgroup of $P$. 
We call it the \textit{standard} Levi group of $P$, and we call $\mathfrak l$ the \textit{standard} Levi subalgebra of $\mathfrak p$.
If $P$ corresponds to a subset $I\subseteq \Delta$ 
as described above, then the subgroup $U$ such that
$$
\mathfrak u=\bigoplus_{\alpha:-\alpha\notin\Phi_I\text{ and }\alpha\in\Phi^+}\mathfrak g_\alpha
$$
is the unipotent radical of $P$, and the subgroup $U^-$ whose Lie algebra equals
$$
\mathfrak u^-=\bigoplus_{\alpha\in\Phi^-\setminus\Phi_I} \mathfrak g_\alpha
$$
is the unipotent radical of $P^-$. Denote the unipotent radical of $B$ by $U_0$, $\mathfrak u_0=\Lie U_0$.

Denote the character lattice of $T$ by $\mathfrak X$. 
Let $\mathfrak X^+$ be the subsemigroup
of dominant weights with respect to $B$.
Denote the fundamental weight corresponding to a simple root $\alpha_i$ by $\varpi_i$. 
Denote the highest-weight representation of a Lie algebra $\mathfrak g$ with highest weight $\lambda$ by 
$V_\mathfrak g(\lambda)$ or by $V(\lambda)$ if it is clear which Lie algebra we are talking about.
Denote by $v_\lambda$ a (unique up to multiplication by a scalar) highest-weight vector of this 
representation. If $V$ is a representation of $\mathfrak g$, denote the set of its weights by $\mathfrak X(V)$.

Denote the Grassmannian of $k$-dimensional subspaces in a vector space $V$ by $Gr(k,V)$. Denote 
the identity operator on a vector space $V$ by $\id_V$, and denote the $k\times k$ identity matrix by $\id_k$.
If a vector space $V$ is equipped with 
a multiplication $V\otimes V\to V$, and $v\in V$, we denote by $\mu_v$ the operator of left multiplication by 
$v$, i.~e. $\mu_v\colon V\to V$, $\mu_v w=vw$.

\section{Automorphisms of generalized flag varieties}\label{autgroups}
\begin{lemma}
Let $X$ be a generalized flag variety. Then its automorphism group in terms of 
Definition \ref{myautdef} exists and is unique up to a unique isomorphism.
\end{lemma}

\begin{proof}
It is known \cite[Chapter V, Theorem 1.4]{fano} that all generalized flag varieties are Fano varieties, i.~e. 
the highest exterior power of the tangent bundle on $X$ (which we denote by $L=\Lambda^mTX$, $m=\dim X$) 
is ample. Moreover, Theorem 1.4 in \cite[Chapter V]{fano} says that 
$L$ is very ample.
Then $L$ defines 
an embedding $\iota$ of $X$ into a projective space $\PP(V)$, where $V=\Gamma(X,L)^*$. Let $G$ be the subgroup 
of all elements of $PGL(V)$ that preserve $\iota(X)$.

Note first that $\iota(X)$ is not contained in any proper projective subspace of $Y$. The 
contrary would mean that there is a non-zero linear function $v$ on $V$ 
(i.~e. an element of $\Gamma(X,L)$) that vanishes 
at all elements of $V$ defined (up to multiplication by a scalar) by 
points of $X$. In other words, $v$ is a section of $L$ such that at each point $x\in X$ 
the value of $v$ in the fiber over $x$ can be obtained by multiplication of a basis vector 
of this fiber by zero. But then $v$ is the zero section.

Hence, since $X$ is irreducible, every proper projective subspace of $\PP(V)$ intersects $\iota(X)$ by a subvariety 
of smaller dimension. If $g\in G$, $g\ne1_G$ acts on $\iota (X)$ trivially, then $\iota(X)$ is 
contained in the union of the projectivizations of the eigenspaces of $g$, i.~e, $\iota(X)$ is 
a union of subvarieties of smaller dimension, and this is impossible. Therefore, $G$ acts on 
$\iota(X)$ faithfully. We are going to prove that $G=\Aut (X)$.

Given an algebraic automorphism of $X$, its differential is an algebraic vector bundle automorphism 
of $TX$. One checks directly that an algebraic action of an algebraic group $H$ on $X$ gives rise to 
an algebraic action of $H$ on $TX$ (and consequently on $L$) by vector bundle automorphisms. 
Since $\Gamma (X,L)$ is finite dimensional, $H$ also acts algebraically on it, and the map 
$\iota\colon X\to\PP(V)$ is $H$-equivariant by construction. The action $H:V$ is linear by 
construction, and for a linear algebraic action it is clear that it yields an algebraic group morphism
$H\to GL(V)\to PGL(V)$. Since $H$ preserves $\iota(X)$, we have a morphism $f\colon H\to G$.
Since $G$ acts on $X$ faithfully, any morphism of abstract groups $g\colon H\to G$ such that 
for every $x\in X$ and $h\in H$ one has $h\cdot x=g(h)\cdot x$ 
coincides with $f$.

If $G'$ is another automorphism group, the isomorphism between $G$ and $G'$ is established 
in the usual way: there exist unique morphisms $f\colon G\to G'$ and $f'\colon G'\to G$ 
such that for every $x\in X$, $g\in G$ and $g'\in G'$ one has $g\cdot x=f'(g')\cdot x$ and 
$g'\cdot x=f(g)\cdot x$. So $f'(f(g))\cdot x=g\cdot x$, but there exists only one 
morphism $h\colon G\to G$ such that $h(g)\cdot x=g\cdot x$, and both $f'\circ f$ and 
the identity automorphism of $G$ are examples of such $h$, so $f'\circ f$ is the identity 
morphism on $G$. Similarly, $f\circ f'$ is the identity morphism on $G'$.
\end{proof}

The existence of the automorphism group reduces the problem of classification 
of generically transitive actions $(\Ga)^m:X$ to the problem of classification
of the subgroups of $\Aut (X)^\circ$ that are isomorphic to $(\Ga)^m$ and 
act generically transitively on $X$.

The book \cite{onisch} deals with smooth (real and complex) 
manifolds rather than algebraic varieties, 
and a different definition of the automorphism group (or, more exactly, 
of the Lie group structure on the abstract automorphism group) of a smooth manifold 
is used there. Namely, a Lie group structure on an (abstract) subgroup $G$
of the (abstract) group of diffeomorphisms of a manifold $M$ is called \cite[Chapter 1, \S2.5]{onisch}
\textit{Lie transformation group} if the action $G:M$ is smooth in terms 
of this Lie group structure and every Lie group action $\mathbb R:M$ 
such that every $r\in\mathbb R$ acts on $M$ with an automorphism that belongs to $G$ 
gives rise to a \textit{smooth} map $\mathbb R\to G$. It is proved in \cite{onisch140}
that such a Lie group structure is unique if it exists. By \cite[Chapter 3, Theorem 1.1]{onisch39}, 
if $M$ is a complex manifold, the group of all biholomorphic automorphisms of $M$ 
is a Lie transformation group (i.~e. the desired Lie group structure exists).
This Lie group is denoted by $\Bih M$.

These definitions and theorems don't say anything about other groups acting on $M$ 
except $\mathbb R$. So first we are going to show that the automorphism group as it is defined 
in Definition \ref{myautdef} equipped with complex topology is a Lie transformation group.
Smoothness of the action is clear, so consider a smooth (with respect to complex topology 
on $X$) action $\mathbb R:X$ by algebraic automorphisms. Every algebraic automorphism of $X$ leads to 
algebraic automorphisms of $TX$ and $L$ and hence to a linear automorphism 
of $\Gamma(X, L)$. So we have a smooth linear action $\mathbb R:V$ and therefore a smooth 
embedding $\mathbb R\to G$.

Now we use the following theorem from \cite{onisch}:
\begin{theorem}\label{onischmainthm}\cite[Chapter 4, \S15.4, Theorem 2]{onisch}
Let $M=G/P$ be a flag manifold for a connected complex semisimple lie group $G$ acting on $M$ 
faithfully. Let $G=G^{(1)}\times\ldots\times G^{(s)}$ be the factorization of $G$ into a product of 
simple subgroups, and $P^{(i)}=G^{(i)}\cap P$. Then $M$ is $G$-isomorphic to $M^{(1)}\times\ldots\times M^{(s)}$,
where $M^{(i)}=G^{(i)}/P^{(i)}$, and this decomposition gives rise to an isomorphism
$$
(\Bih M)^\circ\simeq(\Bih M^{(1)})^\circ\times\ldots\times(\Bih M^{(s)})^\circ.
$$
If $G$ is simple, then $(\Bih M)^\circ$ is simple, and $G=(\Bih M)^\circ$ except for the following cases:
\begin{enumerate}
\item $G=PSp_{2l}(\mathbb C)$, $P=P_1$, $M=\PP^{2l-1}$, $(\Bih M)^\circ=PSL_{2l}(\mathbb C)$.
\item $G$ is of type $G_2$, $P=P_1$, $M=SO_7(\mathbb C)/P_1$, $(\Bih M)^\circ=SO_7(\mathbb C)$.
\item $G=SO_{2l-1}(\mathbb C)$, $P=P_{l-1}$, $M=PSO_{2l}(\mathbb C)/P_l$, $(\Bih M)^\circ=PSO_{2l}(\mathbb C)$.
\end{enumerate}
\end{theorem}
We use this theorem as follows. Let $G$ be a semisimple algebraic group $G$ whose center is trivial, 
and let $P$ be a parabolic subgroup. We know that the group $\Aut (G/P)$ exists and 
is a Lie transformation group being considered with classical topology. By GAGA theorem, 
all biholomorphic automorphisms of $G/P$ are algebraic, so $\Bih (G/P)$ consists of the same 
automorphisms as the abstract automorphism group. Since the Lie transformation group structure is 
unique, $\Bih (G/P)$ as a Lie group coincides with $\Aut (G/P)$ with classical topology.
Since the notion of a connected component is the same for smooth algebraic varieties in Zariski 
topology and in classical topology, this is also true for $\Bih (G/P)^\circ$ and $\Aut (G/P)^\circ$.

Now decompose $G$ into a product of simple groups 
$G=G^{(1)}\times\ldots\times G^{(s)}$ and set $P^{(i)}=P\cap G^{(i)}$. Similarly, we conclude that
$\Aut (G^{(i)}/P^{(i)})^\circ$ with classical topology equals $\Bih (G^{(i)}/P^{(i)})^\circ$.
Clearly, $\prod_i\Aut (G^{(i)}/P^{(i)})^\circ$ can be embedded into $\Aut (G/P)^\circ$ algebraically, 
and we see now from Theorem \ref{onischmainthm} that this is an isomorphism.

Now consider the case when $G$ is simple. First, Theorem \ref{onischmainthm} states that 
if $(G,P)$ is one of the pairs $(PSp_{2l}, P_1)$, $(\text{group of type $G_2$}, P_1)$, 
$(SO_{2l-1}, P_{l-1})$ (these pairs will be called exceptional in what follows), 
then $G/P$ is isomorphic (respectively) to $\PP^{2l-1}=PSL_{2l}/P_1$, 
$SO_7/P_1$ or $PSO_{2l}/P_l$ as a complex manifold, and hence, by GAGA theorem, as an algebraic 
variety. And if $(G,P)$ is not an exceptional pair (note that $(PSL_{2l},P_1)$, 
$(SO_7,P_1)$ and $(PSO_{2l},P_l)$ are not exceptional), then $\Bih (G/P)^\circ$ coincides 
with the set of automorphisms originating from the left action $G:G$. Arguing as above, 
we see that $\Aut (G/P)^\circ$ is the same set of automorphisms.

Therefore, the classification of the actions of $(\Ga)^m$ on $G/P$, where $m=\dim G/P$
and $G$ is semisimple, is now reduced to the problem of classification of the subgroups 
of $\widetilde G$ isomorphic to $(\Ga)^m$ acting generically transitively on $\widetilde G/\widetilde P$, 
where $\widetilde G$ is simple and the pair $(\widetilde G,\widetilde P)$ is not exceptional. 


\section{Reduction from $(\Ga)^m$-actions to multiplications}
\label{reduction}

\begin{proof}[Proof of Proposition \ref{gaactsubalg}]
By definition of an automorphism group, effective $(\Ga)^m$-actions 
on $G/P$ yield embeddings $(\Ga)^m\to(\Aut(G/P))^\circ=G$, 
and a $G$-conjugation of an action corresponds to a conjugation
of the corresponding subgroup in $G$. After a suitable conjugation 
we may suppose that $(\Ga)^m eP\subseteq G/P$ is the open orbit. 
Denote the image of the embedding $(\Ga)^m\hookrightarrow G$ by $A$.
The orbit $(\Ga)^m eP\subseteq G/P$ is open if and only if the 
subset $AP\subseteq G$ is open, and this is equivalent to 
$\mathfrak a+\mathfrak p=\mathfrak g$. 
Since $m=\dim (G/P)$, 
$\mathfrak a+\mathfrak p=\mathfrak g$ implies $\mathfrak a\cap\mathfrak p=0$.

Clearly, if $p\in P$ and $\mathfrak a+\mathfrak p=\mathfrak g$, then 
$(\Ad p)\mathfrak a+\mathfrak p=\mathfrak g$. It suffices to prove that 
if $g$ is an arbitrary element of $G$ such that 
$(\Ad g)\mathfrak a+\mathfrak p=\mathfrak g$, 
then there exists $p\in P$ such that $(\Ad p)\mathfrak a=(\Ad g)\mathfrak a$.
$\mathfrak a$ defines a point in $Gr(m,\mathfrak g)$, and 
the adjoint action $G:\mathfrak g$ leads to an action $G:Gr(m,\mathfrak g)$, 
which we also denote by $\Ad$. Denote the orbit of $\mathfrak a$ in $Gr(m,\mathfrak g)$ 
under this action by $X$. Clearly, the subspaces defined by the points of $X$ are 
Lie subalgebras. Since $PA$ is an open subset of $G$, the subset 
$(\Ad PA)\mathfrak a\subseteq X$ is also open, and, since $(\Ad A)\mathfrak a=\mathfrak a$, 
we conclude that $(\Ad P)\mathfrak a$ is an open subset of $X$.

Denote $A_1=gAg^{-1}$. Then $\Lie A_1=(\Ad g)\mathfrak a$. Then a similar argument 
using $A_1$ instead of $A$ proves that $(\Ad P)((\Ad g)\mathfrak a)$ is an open subset of $X$.
Since $G$ is connected, $X$ is irreducible, and its open subsets 
$(\Ad P)\mathfrak a$ and $(\Ad P)((\Ad g)\mathfrak a)$ intersect nontrivially, i.~e. there 
exist $p_1,p_2\in P$ such that $(\Ad p_1)\mathfrak a=(\Ad p_2g)\mathfrak a$, so 
$(\Ad p_2^{-1}p_1)\mathfrak a=(\Ad g)\mathfrak a$.
\end{proof}

\begin{proof}[Proof of Proposition \ref{subalgisunip}]
Consider again the point in $Gr(m,\mathfrak g)$ defined by $\mathfrak a$ and the action 
$G:Gr(m,\mathfrak g)$. Again denote the orbit of $\mathfrak a$ under this action by $X$.
We already know that $(\Ad P)\mathfrak a$ is an open subset of $X$.

Consider a maximal unipotent subalgebra $\mathfrak u_1$ containing $\mathfrak a$. Since 
all maximal unipotent subgroups are conjugate, there exists $g\in G$ such that 
$(\Ad g)\mathfrak u_1=\mathfrak u_0^-$. Then $(\Ad g)\mathfrak a\subseteq\mathfrak u_0^-$. 
It is clear from Bruhat decomposition of $G$ that $PU_0^-$ is 
an open subset of $G$, so $(\Ad PU_0^-)((\Ad g)\mathfrak a)$ is an open subset of $X$.
Therefore, the subsets $(\Ad P)\mathfrak a$ and $(\Ad PU_0^-)((\Ad g)\mathfrak a)$ have 
a nonempty intersection, i.~e. there exist $p_1,p_2\in P$, $u_0\in U_0^-$ 
such that $(\Ad p_1)\mathfrak a=(\Ad p_2u_0)(\Ad g)\mathfrak a$. Set 
$p=p_2^{-1}p_1$, then $(\Ad p)\mathfrak a=(\Ad u_0)(\Ad g)\mathfrak a\subseteq \mathfrak u_0^-$
since  $(\Ad g)\mathfrak a\subseteq \mathfrak u_0^-$.
\end{proof}

\begin{proof}[Proof of Theorem \ref{actsandmults}]
First, let $\mathfrak a\subseteq \mathfrak u_0^-$ be a commutative unipotent subalgebra 
satisfying $\mathfrak a\cap \mathfrak p=0$ and $\mathfrak a\oplus\mathfrak p=\mathfrak g$.
Consider the map $\varphi\colon \mathfrak u^-\to\mathfrak l$ described in the Introduction.
Since $\mathfrak l$ is a subalgebra of $\mathfrak p$, and $\mathfrak u^-$ is a commutative 
ideal of $\mathfrak p$, we have $0=[u+\varphi (u),v+\varphi (v)]=[u,\varphi (v)]+[\varphi (u),v]+
[\varphi (u),\varphi (v)]$. Here $[u,\varphi (v)]+[\varphi (u),v]\in\mathfrak u^-$, 
$[\varphi (u),\varphi (v)]\in\mathfrak l$, and since $\mathfrak l\cap \mathfrak u^-=0$, 
we get $[u,\varphi (v)]+[\varphi (u),v]=0$ and $[\varphi (u),\varphi (v)]=0$. 
Hence, $[u,\varphi (v)]=[v,\varphi (u)]$, and the multiplication on 
$\mathfrak u^-$ defined in the Introduction is commutative. Consider one 
more element $w\in\mathfrak u^-$. By Jacobi identity, 
$0=[w,[\varphi (u),\varphi (v)]]=
[[w,\varphi (u)],\varphi (v)]+[\varphi (u),[w,\varphi (v)]]=[\varphi (v),[\varphi (u),w]]-
[\varphi (u),[\varphi (v),w]]$. In terms of multiplication this can be written as 
$v(uw)=u(vw)$,
but we already know that the multiplication is commutative, so $(uw)v=u(wv)$, and the 
multiplication is associative. The possibility to write multiplication operators in the 
form $\ad g|_{\mathfrak u^-}$, where $g\in \mathfrak l$, follows directly from the definitions 
of $\varphi$ and of the multiplication. To prove nilpotency of the multiplication 
operators, recall that by Proposition \ref{subalgisunip}, there exists $p\in P$ 
such that $(\Ad p)\mathfrak a\subseteq \mathfrak u_0^-$. The corresponding linear map $\mathfrak u^-\to \mathfrak l$
for the subalgebra $(\Ad p)\mathfrak a$ can be written as $u\mapsto (\Ad p)\varphi ((\Ad p^{-1})u)$, 
so $(\Ad p)\varphi ((\Ad p^{-1})u)\in\mathfrak u_0^-$ and is a nilpotent element of $\mathfrak g$. 
Since the adjoint action preserves the Jordan--Chevalley decomposition of elements of $\mathfrak g$,
$\varphi ((\Ad p^{-1})u)$ is also a nilpotent element of $\mathfrak g$, so $\ad\varphi ((\Ad p^{-1})u)$
is a nilpotent operator for every $u\in\mathfrak u^-$. 
Since $P$ preserves $\mathfrak u^-$,
$\ad\varphi (u)$ and $(\ad\varphi (u))|_{\mathfrak u^-}$
are nilpotent operators for every $u\in\mathfrak u^-$, 
and the latter is exactly the operator of multiplication by $u$.

%
%

Now suppose that we have an $\mathfrak l$-compatible multiplication on $\mathfrak u^-$.
Then every multiplication operator $\mu_u$ ($u\in\mathfrak u^-$) can be written as $\ad g|_{\mathfrak u^-}$
for some $g\in\mathfrak l$, and we set $\varphi (u)=g$. If $a,b\in\mathbb C$, $u, v\in\mathfrak u$, 
then $\mu_{au+bv}=a\mu_u+b\mu_v=(\ad (a\varphi (u)+b\varphi (v)))|_{\mathfrak u^-}$, and, 
since the representation $\mathfrak l:\mathfrak u^-$ is faithful, 
$\varphi (au+bv)=a\varphi (u)+b\varphi (v)$, so $\varphi$ is linear. 
Since the multiplication is commutative, $[\varphi (u),v]=uv=vu=[\varphi (v),u]$.
Since the multiplication is commutative and associative, every two multiplication
operators commute, so, since the representation is faithful, $[\varphi (u),\varphi(v)]=0$.
Using the commutativity of $\mathfrak u^-$, we conclude that 
$[u+\varphi (u),v+\varphi (v)]=[u,\varphi (v)]+[\varphi (u),v]+
[\varphi (u),\varphi (v)]=0$, and the subspace $\mathfrak a=\{u+\varphi (u)\mid u\in\mathfrak u^-\}$
is in fact a commutative subalgebra.

All elements of $\mathfrak l$ of the form $\varphi (u)$ act in its representation 
$\mathfrak u^-$ by nilpotent operators. 
Moreover, since $[\varphi (u),\varphi (v)]=0$ for all 
$u,v\in\mathfrak u^-$, all elements of the form $\varphi (u)$ form 
a (commutative) unipotent subalgebra of $\mathfrak l$. All maximal unipotent subalgebras are conjugate, 
so there exists an element $l\in L$ such that 
$(\Ad l)\varphi (\mathfrak u^-)\subseteq \mathfrak u_0^-\cap\mathfrak l$.
But then $(\Ad l)\mathfrak a\subseteq (\Ad l)\mathfrak u^-+(\Ad l)\varphi (\mathfrak u^-)=
\mathfrak u^-+(\Ad l)\varphi (\mathfrak u^-)\subseteq \mathfrak u_0^-$, so 
$(\Ad l)\mathfrak a$ is a unipotent subalgebra of $\mathfrak g$, and 
$\mathfrak a$ is also a unipotent subalgebra of $\mathfrak g$.
\end{proof}

\section{General facts about $\mathfrak l$-compatible multiplications}
\label{multgeneral}

In this section we fix a reductive group $L$ and its representation $V$. 
Then $V$ is also a representation of $\mathfrak l$, and we denote the 
corresponding morphism of Lie algebras $\mathfrak l\to\mathfrak{gl}(V)$ 
by $\rho$.

\begin{proof}[Proof of Proposition \ref{possibledecomp}]
First, suppose that 
$\mathfrak l$ is a semisimple algebra and
$V$ is an irreducible representation.
\begin{lemma}\label{possibledecompirred}
Let $\mathfrak l$ be a semisimple algebra, $\mathfrak l=\mathfrak l_1\oplus\ldots\oplus\mathfrak l_s$
be its decomposition into a sum of simple summands, and $V$ be an irreducible representation of 
$\mathfrak l$.
Suppose that there exists 
a nontrivial
$\mathfrak l$-compatible multiplication 
on $V$.

Then there exists an index $k$ such that $\mathfrak l_iV=0$ if $i\ne k$.
\end{lemma}
\begin{proof}
Denote by $\rho\colon \mathfrak l\to\mathfrak{gl}(V)$
the corresponding morphism of Lie algebras. 
Every irreducible representation of a semisimple Lie algebra can be written 
as a tensor product of irreducible representations 
of its simple summands, so let $V=V_1\otimes\ldots\otimes V_s$, 
where $V_i$ is an irreducible representation of $\mathfrak l_i$, 
be such a decomposition. Denote by $I$ the set of indices $i$ such 
that $V_i$ is nontrivial. It is sufficient to prove that $I$ contains 
exactly 
one element.

Given a vector $v\in V$, denote by $\varphi (v)$ an element of $\bigoplus_{i\in I}\mathfrak l_i$
such that $vw=\rho(\varphi (v))w$ for every $w\in V$. The remaining simple summands of 
$\mathfrak l$ are in the kernel of $\rho$, so such an element $\varphi (v)$ exists. All representations 
$V_i$ are nontrivial and hence faithful, so $V$ is faithful as a representation 
of $\bigoplus_{i\in I}\mathfrak l_i$. Therefore, $\varphi (v)$ is 
uniquely determined and, as we have seen in the proof of Theorem \ref{actsandmults},
this implies that $\varphi$ is a linear map. It cannot be a zero map, otherwise the 
multiplication would be trivial. Choose an index $k\in I$ such that there exists 
$v\in V$ such that the projection (in terms of the decomposition 
$\mathfrak l=\mathfrak l_1\oplus\ldots\oplus\mathfrak l_s$) of $\varphi (v)$ to $\mathfrak l_k$ 
is not equal to zero. Denote 
$\widetilde {\mathfrak l}=\bigoplus_{i\in I\setminus\{k\}}\mathfrak l_i$, $W=\bigotimes_{i\ne k}V_i$.
Then $V_k$ is a faithful irreducible representation of $\mathfrak l_k$ and 
$W$ is 
an
irreducible representation of $\widetilde{\mathfrak l}$. If $W$ is trivial, then we are done. 
So in the sequel suppose that $\dim W>1$, then $W$ is also faithful.
Denote the corresponding homomorphisms of Lie algebras $\mathfrak l_k\to\mathfrak{gl}(V)$ and 
$\widetilde {\mathfrak l}\to\mathfrak{gl}(W)$ by $\rho_{V_k}$ and $\rho_W$, respectively.

We are going to prove that 
$\varphi(V)\subseteq \mathfrak l_k$.
Assume the contrary. Then there exists a vector $v\in V$ such that $\varphi (v)=x+y$, where 
$x\in \mathfrak l_k$, $y\in\widetilde{\mathfrak l}$ and $x\ne 0$, $y\ne 0$.

Consider now an arbitrary vector $w\in V$ and write $\varphi (w)=x'+y'$, where $x'\in\mathfrak l_k$, 
$y'\in\widetilde{\mathfrak l}$. 
Then $\rho(x'+y')=\rho_{V_k}(x')\otimes\id_W+\id_{V_k}\otimes \rho_W (y')$. 
Since $\mathfrak l_k$ and $\widetilde {\mathfrak l}$ are semisimple Lie algebras, $\tr \rho_{V_k}(x')=\tr \rho_W(y')=0$.
We have the following vector space decomposition of $\mathfrak{gl}(V)$:
$$
\mathfrak{gl} (V) = \mathfrak{gl} (V_k)\otimes\mathfrak{gl} (W)=
\langle\id_{V_k}\rangle\otimes\langle\id_{W}\rangle\oplus
\langle\id_{V_k}\rangle\otimes\mathfrak{sl}(W)\oplus
\mathfrak{sl}(V_k)\otimes\langle\id_{W}\rangle\oplus
\mathfrak{sl}(V_k)\otimes\mathfrak{sl}(W),
$$
and we see that the operator $\rho(x'+y')$ is in the sum of the second and the third summand 
of this decomposition. In other words, for every $w\in V$, the operator of 
multiplication by $w$ is an element of
$\langle\id_{V_k}\rangle\otimes\langle\id_{W}\rangle\oplus
\langle\id_{V_k}\rangle\otimes\mathfrak{sl}(W)\subset \mathfrak{gl}(V)$.

In particular, this holds for $w=v^2$. Since the multiplication on $V$ is associative, 
$\mu_{v^2}=(\rho(x+y))^2=(\id_{V_k}\otimes \rho_W(y)+\rho_{V_k}(x)\otimes \id_W)^2=
\id_{V_k}\otimes (\rho_W(y)^2)+(\rho_{V_k}(x)^2)\otimes \id_W+
2\rho_{V_k}(x)\otimes \rho_W(y)$. We know that $x\ne 0$ and $y\ne 0$ and that 
$\rho_{V_k}$ and $\rho_W$ are injective linear maps, so the last summand is not 
zero and is an element of $\mathfrak {sl}(V_k)\otimes \mathfrak {sl}(W)\subset \mathfrak{gl}(V)$, 
and this is a contradiction.

Thus, $\varphi (V)\subseteq \mathfrak l_k$. 
Denote $d=\dim W$. Recall that we have assumed that $d>1$.
Choose a basis $e_1, \ldots, e_d$ of $W$. Then $V$ can be 
decomposed into a sum of $\mathfrak l_k$-invariant subspaces $V'_j=\langle e_j\rangle\otimes V_k$, 
$1\le t\le d$. These subspaces are isomorphic to $V_k$ as $\mathfrak l_k$-representations.
Assume that there exists an index $j$ and a vector $v\in V'_j$ 
such that 
$\mu_v\ne 0$.
Then $\varphi (v)\ne 0$, 
and $\varphi (v)\in\mathfrak l_k$ acts nontrivially on all subspaces $V_{j'}$, $1\le j'\le d$. 
In particular, there exists an index $j'\ne j$ and a vector $w\in V_{j'}$ such 
that $\varphi (v)w\ne 0$. Since $V_{j'}$ is $\mathfrak l_k$-stable, 
$\varphi (v)w\in V_{j'}$. In terms of multiplication this means that 
$vw\ne 0$, $vw\in V_{j'}$. But $wv=vw$, and by a similar argument we can conclude that 
$wv\in V_j$, and this is a contradiction.
\end{proof}

We are ready to prove Proposition \ref{possibledecomp} in the whole generality. Let $V=V_1\oplus\ldots\oplus V_l$
be the decomposition of $V$ into irreducible summands. Choose two indices $i\ne j$, $1\le i,j\le t$ and 
choose arbitrary $v\in V_i$. Since the multiplication on $V$ is $\mathfrak l$-compatible, 
there exists an element $\varphi(v)\in\mathfrak l$ such that 
$\mu_v=\rho(\varphi(v))$.
Similarly, choose arbitrary
$w\in V_j$ and denote by $\varphi(w)$ an element of $\mathfrak l$ such that 
$\mu_w=\rho(\varphi(w))$.
$V_j$ is an $\mathfrak l$-stable subspace of $V$, so
$vw=\rho(\varphi(v))w\in V_j$. On the other hand, $\rho(\mathfrak l)V_i\subseteq V_i$, so $wv=\rho(\varphi(w))v\in V_i$.
But $vw=wv$, hence $vw=0$. Therefore, $V_iV_j=0$ for each pair of indices $i\ne j$, $1\le j\le t$.
Let $r$ be the number of indices $i$ ($1\le i\le t$) such that $V_iV_i\ne 0$.
Without loss of generality, $V_iV_i\ne 0$ for $1\le i\le r$, and $V_iV_i=0$ for $r<i\le t$.

Denote by $T_0$ the center of $L$. Let $\mathfrak t_0=\Lie(T_0)$. Let $v\in V$, then there 
exists an element $\varphi(v)\in \mathfrak l$ such that 
$\rho(\varphi (v))=\mu_v$.
$\mathfrak l$ as a vector space
can be decomposed into
the direct sum of $[\mathfrak l, \mathfrak l]$ and $\mathfrak t_0$. Let $\varphi(v)=x+y$, where 
$x\in[\mathfrak l,\mathfrak l]$, $y\in\mathfrak t_0$. By Schur's lemma, $y$ acts on each $V_i$ as 
a scalar operator. Since $[\mathfrak l,\mathfrak l]$ is a semisimple algebra, $x$ acts on each $V_i$ 
as an operator with trace 0. $\rho(\varphi(v))$ is a nilpotent operator since it is the operator of multiplication by $v$,
so its restriction to each $V$ also has trace 0. Hence, $\rho(y)|_{V_i}=(\rho(\varphi(v))-\rho(x))|_{V_i}$ is a scalar
operator with trace 0, i.~e. $\rho(y)|_{V_i}=0$ for each $i$, so $\rho(y)=0$, and 
$\rho(x)=\mu_v$.
Therefore, for every $v\in V$ there exists an element of $[\mathfrak l,\mathfrak l]$
that acts on $V$ exactly as 
$\mu_v$.
(Previously we only required 
the existence of such an element in $\mathfrak l$.)

Now we can apply Lemma \ref{possibledecompirred} to the algebra $[\mathfrak l,\mathfrak l]$ and to each of the 
representations $V_i$, $1\le i\le r$. Decompose $[\mathfrak l, \mathfrak l]$ into a direct sum of simple subalgebras,
$[\mathfrak l,\mathfrak l]=\mathfrak l_1\oplus\ldots\oplus\mathfrak l_s$. By Lemma \ref{possibledecomp}, 
for each $i$ ($1\le i\le r$) there exists exactly one index $k$ such that $\mathfrak l_k$ acts 
nontrivially on $V_i$. Choose $v\in V_i$ so that 
$\mu_v\ne 0$.
Then there exists $x_k\in \mathfrak l_k$, $x_k\ne 0$, that acts on $V$ as 
$\mu_v$. (Despite our usual agreement, here $x_k$ do not have to be Chevalley generators corresponding to simple roots.)
Choose an index $j$, $1\le j\le t$, $j\ne i$. The action of $\mathfrak l_k$ on $V_j$ is either trivial or faithful.
If it is faithful, $\rho(x_k)V_j\ne 0$, 
$\rho(x_k)=\mu_v$,
and $V_iV_j=0$, 
so this is a contradiction, and $\mathfrak l_k$ acts trivially on $V_j$. In particular, we see that 
if $i\ne j$, $1\le i,j\le r$, and $\mathfrak l_k$ and $\mathfrak l_{k'}$ are the direct summands that 
act nontrivially on $V_i$ and on $V_j$, respectively, then $k\ne k'$. In other words, 
$s\ge r$ and without loss of generality we may assume that the only summand that acts nontrivially 
on $V_i$ ($1\le i\le r$) is $\mathfrak l_i$. The summands $\mathfrak l_j$ with $j>r$ must act 
trivially on all subrepresentations $V_i$ with $1\le i\le r$, and they may act arbitrarily on subrepresentations
$V_i$ with $r<i\le t$.
\end{proof}

\begin{proof}[Proof of Proposition \ref{shortroot}]
Choose an element $v\in V$ such that $\mu_v\ne 0$. Since $\mu_v$ is a nilpotent operator,
there exists $k\in\mathbb N$ such that $(\mu_v)^k\ne0$, but $(\mu_v)^{2k}=0$. The multiplication is associative, so 
$(\mu_v)^k=\mu_{v^k}$. Hence, there exists $x\in \mathfrak l$ such that $\rho(x)\ne 0$, but $\rho(x)^2=0$.

Until the end of the proof, we use for $\mathfrak l$ the notation for the root system, subalgebras, $\mathfrak{sl}_2$-triples, and the fundamental weights
that we have introduced in Section \ref{prelim} for the Lie algebra of an arbitrary simple group $G$.
Denote the highest root of $\Phi$ by $\alpha$. $x$ is a nilpotent element of $\mathfrak l$, therefore its orbit closure 
contains $y_\alpha$. Hence, $\rho(y_\alpha)^2=0$. In particular, if we decompose $V$ into 
a direct sum of irreducible representations of $\mathfrak{sl}_2=\langle x_\alpha,y_\alpha,h_\alpha\rangle$, then 
the dimension of any direct summand is at most 2. So, the only possible eigenvalues of $\rho(h_\alpha)$ are $-1$, 0, and 1.
In particular, $\lambda(h_\alpha)$ is the eigenvalue of $\rho(h_\alpha)$ corresponding to the highest weight subspace 
of $V$, so $\lambda(h_\alpha)$ can be equal to 1 or 0. Write $\lambda=\sum a_i\varpi_i$. Note that all 
coefficients $a_i$ cannot vanish simultaneously, otherwise $V$ is a trivial representation, $\rho(\mathfrak l)=0$, and 
the multiplication on $V$ has to be trivial.

Fix an invariant scalar multiplication $(\cdot,\cdot)$ on $\mathfrak l$. It identifies $\mathfrak t$ and $\mathfrak t^*$, 
and if $\beta\in \Phi$, then $h_\beta$ is identified with
$$
\frac{2\beta}{(\beta,\beta)}.
$$
So, all vectors $h_\beta$ (for all roots $\beta\in \Phi$) form a root system dual to $\Phi$ in $\mathfrak t$. The set of vectors 
$h_\beta$ for simple positive roots $\beta$ can be chosen as a simple root set for this dual root system.
After this choice, $h_\alpha$ becomes the highest short root of this system. Write $h_\alpha=\sum b_jh_{\alpha_j}$. 
All $b_i$ are positive integers (see \cite[Section 12.2, Table 2]{humps}). We have 
$\lambda(h_\alpha)=\sum a_i\varpi_i(h_\alpha)\le 1$. Recall that by the definition of a fundamental weight, 
$\varpi_i (h_{\alpha_j})=\delta_{ij}$, so $\lambda(h_\alpha)=\sum a_ib_i$.
Therefore, exactly one of the coefficients $a_i$ is nonzero, this coefficient $a_j$ must be 1 (i.~e. $\lambda=\varpi_j$), 
and the index $j$ must satisfy $b_j=1$.
\end{proof}

\section{Existence of $\mathfrak l$-compatible multiplications}
\label{multcasebycase}
In this section, $\mathfrak l=\Lie L$ is a simple Lie algebra, and we use for it the notation we introduced in Section \ref{prelim}
for an arbitrary simple Lie algebra $\mathfrak g$. Let $V$ be an irreducible representation of $L$ satisfying the conditions of 
Proposition \ref{shortroot}. Denote the corresponding morphism of Lie algebras $\mathfrak l\to \mathfrak{gl}(V)$ by $\rho$.

When we prove that for a particular simple algebra and its irreducible representation, nontrivial $\mathfrak l$-compatible 
multiplications do not exist, we will use one of the following two approaches.

\textit{First,} any multiplication $V\times V\to V$ is determined by its structure constant tensor, which is an element 
of $V^*\otimes V^*\otimes V$. A multiplication is commutative if the structure constant tensor belongs to 
$S^2(V^*)\otimes V\subseteq V^*\otimes V^*\otimes V$. Denote by $R(\mathfrak l)$ the algebra $\mathfrak l$ understood as 
the adjoint representation of $\mathfrak l$. Then $\rho\colon R(\mathfrak l)\to \mathfrak{gl}(V)=V^*\otimes V$ is an $\mathfrak l$-equivariant
homomorphism of representations. Hence, $V^*\otimes V^*\otimes V$ contains a subrepresentation isomorphic to $V^*\otimes R(\mathfrak l)$.
The condition requiring that each (left) multiplication operator on $V$ coincide with an operator of the form $\rho(x)$ ($x\in\mathfrak l$)
means in these terms that the structure constant tensor is an element of this subrepresentation.
Therefore, if the subspaces $S^2(V^*)\otimes V$ and $V^*\otimes R(\mathfrak l)$ intersect trivially in $V^*\otimes V^*\otimes V$,
then there are no nontrivial $\mathfrak l$-compatible multiplications on $V$.

To prove that these subspaces intersect trivially, one can first decompose $V^*\otimes R(\mathfrak l)$ 
into a sum of irreducible subrepresentations and then check that highest weight vectors of
all these irreducible subrepresentations are outside $S^2(V^*)\otimes V$. If this is true, then 
the whole subrepresentations intersect $S^2(V^*)\otimes V$ trivially. If, additionally, 
there are no isomorphic representations among them, then by Schur's lemma, 
$V^*\otimes R(\mathfrak l)\cap S^2(V^*)\otimes V=0$. In other words, we have the following lemma:

\begin{lemma}\label{tripletensor}
Decompose $V^*\otimes R(\mathfrak l)$ into 
a sum of irreducible representations. Suppose that there are no isomorphic subrepresentations 
among them. Let $v_i$ be highest weight vectors of these subrepresentations, and let 
$w_i$ be their images under the embedding 
$V^*\otimes R(\mathfrak l)\hookrightarrow V^*\otimes V^*\otimes V$ described above.
If $w_i\notin S^2(V^*)\otimes V\subseteq V^*\otimes V^*\otimes V$ for all $i$, then 
there are no nontrivial $\mathfrak l$-compatible multiplications on $V$.\qed
\end{lemma}

To test whether a vector $w_i$ is an element of $S^2(V^*)\otimes V$, one can consider the canonically corresponding 
map $V^*\to V^*\otimes V^*$. $w_i\in S^2(V^*)\otimes V$ if an only if the image of this map is a subspace of 
$S^2(V^*)$.

\textit{Second,} one can argue as follows. Since the representation $\mathfrak l\colon V$ is faithful, 
given a vector $v\in V$, there exists exactly one element $x\in\mathfrak l$ such that $\rho(x)=\mu_v$.
As in the second part of the proof of Theorem \ref{actsandmults}, we denote this $x$ by $\varphi(v)$.
Again, 
for all $v,w\in V$, $a,b\in \mathbb C$ we have 
$\rho(\varphi (av+bw))=\mu_{av+bw}=a\mu_v+b\mu_w=a\rho(\varphi(v))+b\rho(\varphi(w))$, and
since the representation is faithful, $\varphi(av+bw)=a\varphi(v)+b\varphi(w)$, so $\varphi$ is a 
linear map. Since the multiplication is commutative and associative, $\varphi(V)$ is a commutative 
subalgebra in $\mathfrak l$, and since all multiplication operators are nilpotent, 
$\varphi(V)$ is a unipotent subalgebra. After a suitable conjugation by an element of $L$ 
we may suppose that $\varphi(V)$ is a subalgebra of a prefixed maximal unipotent subalgebra. We will suppose that 
$$
\varphi(V)\subseteq \mathfrak u_0=\bigoplus_{\alpha\in\Phi^+}\mathfrak l_\alpha.
$$

On the other hand, given any linear map $\varphi\colon V\to \mathfrak u_0$, one can define 
a multiplication on $V$ by $vw=\rho(\varphi (v))w$. Choose bases in $V$ and in $\mathfrak u_0$,
then linear maps between $V$ and $\mathfrak u_0$ are determined by $(\dim V)\times (\dim\mathfrak u_0)$-matrices.
$\rho|_{\mathfrak u_0}$ can also be written as a (fixed) element of $(\mathfrak u_0)^*\otimes V^*\otimes V$.
A multiplication is commutative if and only if $\rho(\varphi(v))w=\rho(\varphi(w))v$ for all $v,w\in V$. This equation is bilinear 
in $v$ and $w$, so it is sufficient to satisfy it for $v$ and $w$ being elements of the basis of $V$ we have chosen.
And for fixed $v$ and $w$ this equation can be seen as a linear equation in the 
entries
of the matrix defining $\varphi$.
A multiplication is associative if and only if 
$\rho(\varphi (\rho(\varphi (u))v))w=\rho(\varphi (u))\rho(\varphi (v)) w$ for all $u,v,w\in V$.
Again, this equation is trilinear in $u,v,w$, and if $u,v,w$ are fixed, this is a homogeneous equation
of degree 2 in the coefficients of the matrix defining $\varphi$. All operators of the form $\rho(x)$ with $x\in\mathfrak u_0$ 
are nilpotent. Therefore, we have identified the set of $\mathfrak l$-compatible multiplications on $V$ 
such that $\mu_v\in\rho(\mathfrak u_0)$ for all $v\in V$ with a closed cone in $V^*\otimes \mathfrak u_0$.
Denote this cone by $X$.

The normalizer of $\mathfrak u_0$ in $L$ (denote it by $B$) acts canonically on $V^*\otimes \mathfrak u_0$. The identification 
between the space of structure constant tensors of multiplications on $V$ such that all (left) multiplication operators belong 
to $\rho(\mathfrak u_0)$ and $V^*\otimes \mathfrak u_0$ described above is equivariant under this action.
The conditions in the definition of $\mathfrak l$-compatibility are $\mathfrak l$-invariant, therefore 
$B$ preserves $X$. Hence, $B$ acts on the projectivization of $X$ in $\mathbf P(V^*\otimes \mathfrak u_0)$.
$B$ is a Borel subgroup of $L$, so its action on a projective 
variety always has a fixed point, in particular, there is a $B$-fixed point in $X$.
So, $X$ contains a highest weight vector for an irreducible subrepresentation of the representation of $L$ in $V^*\otimes R(\mathfrak l)$.
Therefore, if there exists a nontrivial $\mathfrak l$-compatible multiplication on $V$, then 
there exists a nontrivial $\mathfrak l$-compatible multiplication on $V$ such that 
the map $\varphi\colon V\to \mathfrak l$ constructed above maps $V$ to $\mathfrak u_0$ 
and is a highest weight vector in an irreducible subrepresentation of the $\mathfrak l$-module $V^*\otimes R(\mathfrak l)$. In the sequel we suppose 
that $\varphi$ satisfies these two conditions. 
Denote the weight of $\varphi$ in $V^*\otimes R(\mathfrak l)$
by $\kappa$.

Let $-\lambda^*$ be the lowest weight 
of $V$, $v_{-\lambda^*}\in V$ be a lowest weight vector. 
Then $\varphi(v_{-\lambda^*})\in\mathfrak l_\gamma$, where $\gamma=\kappa-\lambda^*$.

Suppose first that $\varphi(v_{-\lambda^*})=0$. Let us prove that in this case $\varphi=0$.
Indeed, $\varphi$ is a highest weight vector in $V^*\otimes\mathfrak l$, so $\mathfrak u_0\cdot \varphi=0$.
This means that for every $u\in\mathfrak u_0$ and for every $v\in V$, we have $-\varphi(\rho(u)v)+(\ad u)(\varphi(v))=0$.
In other words, we have the following equality of linear maps from $V$ to $\mathfrak l$: $(\ad u)\circ \varphi=\varphi\circ\rho (u)$.
As a vector space, $V$ is generated by the images of $v_{-\lambda^*}$ under arbitrary products of operators 
of the form $\rho(u)$, where $u\in \mathfrak u_0$. Then $\varphi(V)$ is generated by the images of 
$\varphi (v_{-\lambda^*})$ under products of operators of the form $\ad(u)$, $u\in \mathfrak u_0$. 
But $\varphi (v_{-\lambda^*})=0$, so $\varphi(V)=0$, i.~e. $\varphi=0$.

Now suppose that $\varphi(v_{-\lambda^*})\ne 0$. 
Then $\gamma\in\Phi\cup\{0\}$. Moreover, in fact $\gamma\in\Phi^+$ since we have supposed that $\varphi(V)\subseteq\mathfrak u_0$.
Under the assumptions we made, we prove the following lemma.

\begin{lemma}\label{weightlatticefailure}
Let $\mathfrak l$ be a simple Lie algebra, $\lambda\in\mathfrak X^+$ be a dominant weight and $V=V_{\mathfrak l}(\lambda)$. Suppose that there exists a nontrivial 
$\mathfrak l$-compatible multiplication on $V$. 
Then there exists $\gamma\in \Phi^+$ such that 
\begin{enumerate}
\item $\gamma+\lambda^*$ is the highest weight
of an irreducible subrepresentation of $V^*\otimes R(\mathfrak l)$.
\item There exist \textbf{no} weights $\nu\in\mathfrak X(V)$ such that $\gamma+\nu\in\mathfrak X(V)$ and $\nu+\lambda^*+\gamma\notin\Phi^+$.
\end{enumerate}
\end{lemma}

\begin{proof}
Assume the contrary, i.~e. assume that there exists a weight $\nu\in\mathfrak X(V)$ such that 
$\gamma+\nu\in\mathfrak X(V)$ and $\nu+\lambda^*+\gamma\notin\Phi^+$.
Denote the corresponding weight space by $V_\nu\subseteq V$.
Since $\gamma\in\Phi^+$ is a positive root such that $\gamma+\nu\in\mathfrak X(V)$, 
$\mathfrak{sl}_2$ representation theory implies that
$\ker(\rho(\mathfrak l_\gamma)|_{V_\nu})$ is a subspace of 
codimension 1 in $V_\nu$. 
Choose 
an arbitrary vector $w_\nu\in V_\nu$ outside this kernel. 

$\varphi$ is a vector of weight $\kappa$ in $V^*\otimes\mathfrak l$, so $\varphi(V_\nu)\subseteq \mathfrak l_{\nu+\kappa}=
\mathfrak l_{\nu+\lambda^*+\gamma}$. But $\nu+\lambda^*+\gamma\notin\Phi^+$, so $\varphi(V_\nu)=0$.
In particular, $w_\nu v_{-\lambda^*}=\rho(\varphi (w_\nu))v_{-\lambda^*}=0$.
On the other hand, $w_\nu v_{-\lambda^*}=\rho(\varphi (v_{-\lambda^*}))w_\nu
\ne 0$ according to our choice
of $w_\nu$. This is a contradiction.
\end{proof}

Now we are going to consider types of simple Lie algebras 
and the corresponding fundamental weights satisfying Proposition \ref{shortroot}
case by case. If there exists a diagram automorphism of a Lie algebra that interchanges two fundamental weights, 
we consider only one of them.

\subsection{Algebra $\mathfrak l$ of type $A_l$}

It is sufficient to consider fundamental weights $\varpi_1$, $\varpi_{l-2}$, and $\varpi_p$, where $2<p\le \lceil l/2\rceil$.

We need an explicit description for a root system of type $A_l$.
Consider a Euclidean space $E$ with an orthonormal basis $\{\widetilde{\varepsilon_i}\}$
($1\le i\le l+1$), its subspace
$\langle\widetilde{\varepsilon_1}+\ldots+\widetilde{\varepsilon_{l+1}}\rangle$, 
and the orthogonal complement to this subspace. Denote the 
orthogonal
projection $E\to \langle\widetilde{\varepsilon_1}+\ldots+\widetilde{\varepsilon_{l+1}}\rangle^\bot$
by $q$. 
Then vectors $\varepsilon_i=q(\widetilde{\varepsilon_i})$ satisfy $(\varepsilon_i,\varepsilon_i)=l/(l+1)$ and $(\varepsilon_i,\varepsilon_j)=-1/(l+1)$.
We construct a root system of type $A_l$ as follows. 
$\Phi=\{\varepsilon_i-\varepsilon_j\mid 1\le i,j\le l, i\ne j\}$.
For a set of positive roots we can take 
$\Phi^+=\{\varepsilon_i-\varepsilon_j\mid 1\le i<j\le l\}$, 
then $\Delta=\{\varepsilon_i-\varepsilon_{i+1}\mid 1\le i\le l\}$.
The corresponding fundamental weights can be expressed as
$\varpi_i=\varepsilon_1+\ldots+\varepsilon_i$ ($1\le i\le l$).

\subsubsection{$\lambda=\varpi_1$}
\label{commalgarises}
Let $SL_{l+1}$ act in its tautological representation $V$ and preserve a highest degree skew-symmetric form $\omega\ne 0$ on $V$.
A nilpotent operator on $V$ always has trace 0, so 
the only essential conditions a multiplication on $V$ should satisfy to be $\mathfrak l$-compatible are commutativity, associativity, 
and multiplication operator nilpotency. To classify $\mathfrak l$-compatible multiplications up to the action of $SL_{l+1}$,
we have to consider two cases.

Case 1. The action of $SL_{l+1}$ enables one to multiply the structure constant tensor of the multiplication in question
by any complex number. Then the $SL_{l+1}$-orbit of this multiplication coincides with its $GL_{l+1}$-orbit, since the 
central torus of $GL_{l+1}$ can only multiply the structure constant tensor by a scalar. These multiplications 
are in one-to-one correspondence with isomorphism classes of commutative associative $l+1$-dimensional algebras 
with all multiplication operators being nilpotent.

Case 2. Using the action of $SL_{l+1}$, one can only multiply the structure constant tensor of the multiplication in question 
by finitely many complex numbers. Then the $GL_{l+1}$-orbit of this multiplication consists of infinitely many 
$SL_{l+1}$-orbits, which
can be parametrized as follows.
Given an associative commutative $l+1$-dimensional algebra $A$ with all multiplication operators being nilpotent, 
choose a highest degree skew-symmetric form $\nu$ on $A$ and identify $A$ with the vector space $V$ so that $\nu$ is identified 
with $\omega$. This condition does not determine an isomorphism between $A$ and $V$ uniquely, and possible isomorphisms
differ exactly by the action of elements of $SL_{l+1}$. Therefore, in this case multiplications are in one-to-one correspondence 
with isomorphism classes of pairs of a commutative associative $l+1$-dimensional algebra $A$ with nilpotent multiplication 
operators and a nonzero highest degree skew-symmetric form on $A$, where an isomorphism preserves the form as well as the multiplication.

Observe that the possibility to multiply the structure constant tensor by any complex number depends only on the isomorphism class 
of $l+1$-dimensional algebras, it does not depend on the isomorphism we choose between an algebra and the $SL_{l+1}$-module $V$.
So, given such an isomorphism class of algebras, one can determine whether Case 1 or Case 2 takes place and whether it is necessary 
to choose a highest degree skew-symmetric form on an algebra from this class.

Now we are ready to give a detailed description of $L$-isomorphism classes of multiplications
in cases when $L$ is a reductive but not necessarily simple algebraic group, and the decomposition of 
$[\mathfrak l,\mathfrak l]$ into a sum of simple summands contains several components of type $A$.
To define a multiplication on $V$ up to the $L$-action, we first choose the irreducible components of $V$
where a simple subalgebra of type $A$ acts nontrivially (see Proposition \ref{possibledecomp}) 
and where we are going to define a non-zero multiplication such that Case 2 from the above classification will hold, 
i.~e. where it will only be possible to multiply the structure constant tensor by finitely many complex numbers.
Denote these components by $V_1,\ldots,V_k$, and
let $n_1,\ldots,n_k$ be their dimensions. Choose $k$ commutative associative algebras with 
nilpotent multiplication operators of dimensions $n_1,\ldots,n_k$ so that Case 2 holds for each of them.
Let the central torus of $L$ act on $V_i$ via a character $\chi_i$. Then it acts on $\Lambda^{n_i}V_i^*$
via $-n_i\chi_i$. We have to choose $k$ nonzero highest degree skew-symmetric forms up to the action of the torus,
i.~e. we have to choose an orbit of the torus in a $k$-dimensional space
$$
W=\oplus_{i=1}^k\Lambda^{n_i}(V_i^*),
$$
so that each coordinate of (every) point of the orbit is nonzero.
Such orbits are parametrized by values of a tuple of algebraically independent Laurent monomials that 
generate the lattice of all Laurent monomials in the coordinates on $W$ invariant under the action of the torus.

For example, if $L=GL_n$ and $V$ is a tautological representation of $V$ (we will also see this case later),
then the central torus is one-dimensional, and it acts transitively on the set of nonzero vectors of $\Lambda^nV^*$.
Hence, there are no nontrivial invariant Laurent monomials, and the $\mathfrak{gl}_n$-compatible 
multiplications up to the action of $GL_n$ are in one-to-one correspondence 
with isomorphism classes of $n$-dimensional associative commutative algebras with nilpotent multiplication operators.

\begin{example}
An example of $\mathfrak{sl}_{18}$-compatible multiplication such that it is only possible to multiply 
the structure constant tensor by finitely many scalars.
\end{example}

Consider the subalgebra (without unity) in $\mathbb C[x,y]/(x^5+y^5-x^3y^3,x^4y,xy^4)$ generated by $x$ and $y$. Denote it by $A$.
This is an algebra of dimension 18, it has the following basis: $x, y, x^2, xy, y^2, x^3, x^2y, xy^2, y^3, x^4, x^3y, x^2y^2,
xy^3, y^4, x^2y^3, x^3y^2, x^5, y^5$. In what follows, by the degree of a monomial we understand 
its total degree in $x$ and $y$. Consider an automorphism from the identity component 
of the group of automorphisms of $A$.

Every automorphism of $A$ is determined by the images of $x$ and $y$.
Suppose that $x$ is mapped to $ax+by+(\text{terms of higher degree})$ and $y$ is mapped to $cx+dy+
(\text{terms of higher degree})$. The matrix
$$
\left(\begin{array}{cc}
a & c \\ 
b & d
\end{array}\right)
$$
cannot be degenerate, otherwise the intersection of the image of the automorphism and 
the subspace spanned by $x$ and $y$ is at most one-dimensional. This automorphism 
maps $x^5$ to $a^5x^5+10a^3b^2x^3y^2+10a^2b^3x^2y^3+b^5y^5+\alpha x^3y^3=
(a^5+\alpha)x^5+10a^3b^2x^3y^2+10a^2b^3x^2y^3+(b^5+\alpha)y^5$, 
where $\alpha$ is a complex number, 
$y^5$ is mapped to 
$c^5x^5+10c^3d^2x^3y^2+10c^2d^3x^2y^3+d^5y^5+\beta x^3y^3=
(c^5+\beta)x^5+10c^3d^2x^3y^2+10c^2d^3x^2y^3+(c^5+\beta)y^5$, 
where $\beta\in\mathbb C$, and 
$x^3y^3$ is mapped to $\gamma x^3y^3=\gamma x^5+\gamma y^5$,
where $\gamma\in\mathbb C$. The sum of these three monomials must equal zero in $A$.
In particular, $a^3b^2=-c^3d^2$ and $a^2b^3=-c^2d^3$. Assume first that there are no zeros among 
$a, b, c, d$. Then $(a^3b^2)/(a^2b^3)=(-c^3b^2)/(-c^2d^3)$, and $a/b=c/d$. Hence, $ac=bd$, and 
the matrix above is degenerate. Therefore, at least one of the numbers $a, b, c, d$ equals zero. But if
$c=0$ or $d=0$, then $a^3b^2=0$, so $a=0$ or $b=0$. So we conclude that at least one of the numbers $a, b$
equals zero. Then $-c^2d^3=0$, and at least one of the numbers $c,d$ equals zero. If $a=0$ and $c=0$, then 
the matrix is degenerate, and if $b=0$ and $d=0$, the matrix is also degenerate. The only two 
remaining possibilities are $a=d=0$ or $b=c=0$. These two sets of automorphisms of $A$ 
are disjoint, so the identity component of the group of automorphisms is a subset of 
one of them. For the identity automorphism we have $b=c=0$, so these equalities also hold 
for any automorphism from the identity component. Note that $a\ne 0$ and $d\ne 0$, otherwise 
the matrix is degenerate.

Now we can write that $x$ is mapped to $ax+a_1x^2+a_2xy+a_3y^2+(\text{terms of higher degree})$, and $y$ is mapped to 
$dy+d_1y^2+d_2xy+d_3x^2+(\text{terms of higher degree})$. Observe that each monomial of degree at least 7 equals zero in $A$. Indeed, such a monomial 
is divisible by one of the following monomials: $x^4y$, $xy^4$, $x^7$, or $y^7$. The first two equal 
zero in $A$ by the definition of $A$, and for $x^7$ we have: $x^7=-x^2y^5+x^5y^3=0$ since $xy^4=x^4y=0$.
The calculation for $y^7$ is similar. Hence, the image of $x^5$ (resp. $y^5$) does not depend on the terms of degree at least 3 in the image of $x$
(resp. $y$). So, $x^5$ is mapped to $a^5x^5+5a^4x^4(a_1x^2+a_2xy+a_3y^2)=a^5x^5+5a^4a_1x^6=a^5x^5+5a^4a_1(-xy^5+x^4y^3)=a^5x^5$.
Similarly, $y^5$ is mapped to $d^5y^5$. $x^3y^3$ is mapped to $a^3d^3x^3y^3=a^3d^3x^5+a^3d^3y^5$ (the other terms are of degree at least 7).
Hence, $a^5x^5+d^5y^5-a^3d^3x^5-a^3d^3y^5=0$.
$x^5$ and $y^5$ are elements of the basis of $A$ we chose, so $a^5=a^3d^3=d^5$. Hence, $a^2=d^3$, $a^3=d^2$, and $a^2/a^3=d^3/d^2$. In other words, 
$a^{-1}=d$. Now we can write $a^5=1$, and, since the automorphism under consideration belongs to the identity component of the 
automorphism group, $a=1$. So, $d=1$, and the matrix of the automorphism with respect to the basis we chose is 
lower unitriangular. In other words, the identity component of the automorphism group of $A$ is inside $SL_{18}$.

Now suppose that $h\in SL_{18}$ multiplies the structure constant tensor by $t\in\mathbb C$ ($t\ne 0$).
There exists a scalar matrix $g\in GL_{18}$ that also multiplies the structure constant tensor by $t$, namely, $g=t^{-1}\id_A$.
Then $gh^{-1}$ stabilizes the structure constant tensor, i.~e. $gh^{-1}$ is an automorphism of $A$.
Denote the number of connected components of the automorphism group by $k$. Then $(gh^{-1})^k$ is an element of the identity component of 
the automorphism group. So, since $g$ is a scalar matrix, $g^kh^{-k}\in SL_{18}$, $g^k\in SL_{18}$, 
$g^{18k}=\id_A$, and $t$ is a root of unity of degree $18k$. There are only finitely many possibilities for $t$.

\begin{example}
An example of $\mathfrak{sl}_2$-compatible multiplication such that it is possible to multiply 
the structure constant tensor by any complex number.
\end{example}

Consider an algebra with a basis $\{x,y\}$ and the multiplication defined by 
$x^2=y$, $xy=yx=x^2=0$. Clearly, this multiplication is $\mathfrak{sl}_2$-compatible.
The linear operator defined by $x\mapsto tx$, $y\mapsto t^{-1}y$ has determinant 1 and
multiplies the structure constant tensor by $t^{-3}$.

\subsubsection{$\lambda=\varpi_{l-2}$, $l>2$}

We prove that nontrivial $\mathfrak l$-compatible multiplications do not exist using Lemma \ref{tripletensor}.
Set $\mathfrak l=\mathfrak{sl}_{l+1}$, $V^*=\Lambda^2W$, where $W$ is a tautological $\mathfrak{sl}_{l+1}$-module.
Let $e_1,\ldots, e_{l+1}$ be a basis of $W$. 
Unless this leads to ambiguity, we use the same notation for linear operators on $W$ and for elements of $\mathfrak l$.
We have the following Chevalley generators of $\mathfrak l=\mathfrak{sl}(W)$:
$x_i=e_i\otimes e_{i+1}^*$, $y_i=e_{i+1}\otimes e_i^*$, and $h_i=e_i\otimes e_i^*-e_{i+1}\otimes e_{i+1}^*$.
We also have the following basis of $V^*$: $\{e_i\wedge e_j\mid 1\le i<j\le l+1\}$.
The embedding $R(\mathfrak l)\hookrightarrow V^*\otimes V$ maps $e_i\otimes e_j\in\mathfrak l$ to 
$\sum_{k\ne i,j} e_i\wedge e_k\otimes (e_j\wedge e_k)^*$.

From \cite[Table 5]{vinonisch}, we see that $V^*\otimes R(\mathfrak l)\cong
V(\varpi_1+\varpi_2+\varpi_l)\oplus V(2\varpi_1)\oplus V(\varpi_3+\varpi_l)\oplus V(\varpi_2)$. Let us 
find highest weight vectors of the irreducible subrepresentations.

1. Clearly, $e_1\wedge e_2\otimes e_1\otimes e_{l+1}^*\in V^*\otimes R(\mathfrak l)$
is a vector of weight $2\veps_1+\veps_2-\veps_{l+1}=\varpi_1+\varpi_2+\varpi_l$, so it is a highest weight vector in 
$V(\varpi_1+\varpi_2+\varpi_l)\subset V^*\otimes R(\mathfrak l)$.
The embedding $R(\mathfrak l)\hookrightarrow V^*\otimes V$ maps this vector to
$$
\sum_{k=2}^l (e_1\wedge e_2)\otimes (e_1\wedge e_k)\otimes (e_{l+1}\wedge e_k)^*,
$$
and this is not an element of $S^2(V^*)\otimes V$ since $l>2$.

2. $v=\sum_{i=2}^{l+1}e_1\wedge e_i\otimes e_1\otimes e_i^*$ is annihilated by $\mathfrak u_0$.
Indeed, $x_j$ annihilates all summands except for the ones with $i=j$ and $i=j+1$, and 
$e_1\wedge e_j\otimes e_1\otimes e_j^*+e_1\wedge e_{j+1}\otimes e_1\otimes e_{j+1}^*$ is moved 
to $-e_1\wedge e_j\otimes e_1\otimes e_{j+1}^*+e_1\wedge e_j\otimes e_1\otimes e_{j+1}^*=0$.
$v$ is a vector of weight $2\veps_1+\veps_i-\veps_i=2\varpi_1$, so it is a highest weight vector
of the subrepresentation $V(2\varpi_1)\subset V^*\otimes R(\mathfrak l)$. The embedding 
$R(\mathfrak l)\hookrightarrow V^*\otimes V$ maps this vector to
\begin{multline*}
\sum_{i=2}^{l+1} (e_1\wedge e_i)\otimes \sum_{j\ne i, 1<j\le l+1}(e_1\wedge e_j)\otimes (e_i\wedge e_j)^*\\
=\sum_{1<i<j\le l+1}(e_1\wedge e_i\otimes e_1\wedge e_j-e_1\wedge e_j\otimes e_1\wedge e_i)\otimes
(e_i\wedge e_j)^*\in\Lambda^2(V^*)\otimes V,
\end{multline*}
and $\Lambda^2(V^*)\otimes V\cap S^2(V^*)\otimes V=0$.

3. Let us check that $v=e_1\wedge e_3\otimes e_2\otimes e_{l+1}^*-e_2\wedge e_3\otimes e_1\otimes e_{l+1}^*
-e_1\wedge e_2\otimes e_3\otimes e_{l+1}^*$ is annihilated by $\mathfrak u_0$. Clearly, $v$ is annihilated by 
all $x_i$ with $i>2$. $x_1$ moves $v$ to $e_1\wedge e_3\otimes e_1\otimes e_{l+1}^*-
e_1\wedge e_3\otimes e_1\otimes e_{l+1}^*-e_1\wedge e_1\otimes e_3\otimes e_{l+1}^*=0$, and 
$x_2$ moves $v$ to $e_1\wedge e_2\otimes e_2\otimes e_{l+1}^*-
e_2\wedge e_2\otimes e_1\otimes e_{l+1}^*-e_1\wedge e_2\otimes e_2\otimes e_{l+1}^*=0$.

$v$ is a vector of weight $\veps_1+\veps_2+\veps_3-\veps_{l+1}=\varpi_3+\varpi_l$, so it is a highest weight vector in 
$V(\varpi_3+\varpi_l)\subset V^*\otimes R(\mathfrak l)$. 
To check that the image of $v$ under the embedding 
$R(\mathfrak l)\hookrightarrow V^*\otimes V$ 
is not an element of $S^2(V^*)\otimes V$, we check that it defines a map $V^*\to V^*\otimes V^*$ whose image is not in 
$S^2(V^*)$. Indeed, by applying this map to $e_l\wedge e_{l+1}\in V^*$, we get 
$(e_1\wedge e_3)\otimes (e_l\wedge e_2)-
(e_2\wedge e_3)\otimes (e_l\wedge e_1)-(e_1\wedge e_2)\otimes (e_l\wedge e_3)\notin S^2(V^*)$. (For $l=3$, we have 
$(e_1\wedge e_3)\otimes (e_l\wedge e_2)-
(e_2\wedge e_3)\otimes (e_l\wedge e_1)-(e_1\wedge e_2)\otimes (e_l\wedge e_3)=(e_1\wedge e_3)\otimes (e_3\wedge e_2)-
(e_2\wedge e_3)\otimes (e_3\wedge e_1)\in\Lambda^2(V^*)$.)

4. We are going to check that the following vector is annihilated by $\mathfrak u_0$:
$$
v=\sum_{i=3}^{l+1}(e_1\wedge e_i\otimes e_2\otimes e_i^*
-e_2\wedge e_i\otimes e_1\otimes e_i^*+\frac{l-3}2e_1\wedge e_2\otimes e_i\otimes e_i^*)
+\frac{l-1}2 e_1\wedge e_2\otimes (e_1\otimes e_1^* + e_2\otimes e_2^*).
$$
Indeed, $x_j$ with $j\ge 3$ annihilates each individual summand except the ones with $i=j$ and with $i=j+1$, 
and it brings 
\begin{multline*}
(e_1\wedge e_j\otimes e_2\otimes e_j^*
-e_2\wedge e_j\otimes e_1\otimes e_j^*+\frac{l-3}2e_1\wedge e_2\otimes e_j\otimes e_j^*)\\
+(e_1\wedge e_{j+1}\otimes e_2\otimes e_{j+1}^*
-e_2\wedge e_{j+1}\otimes e_1\otimes e_{j+1}^*+\frac{l-3}2e_1\wedge e_2\otimes e_{j+1}\otimes e_{j+1}^*)
\end{multline*}
to
\begin{multline*}
-e_1\wedge e_j\otimes e_2\otimes e_{j+1}^*+e_2\wedge e_j\otimes e_1\otimes e_{j+1}^*
-\frac{l-3}2e_1\wedge e_2\otimes e_j\otimes e_{j+1}^*\\
+e_1\wedge e_j\otimes e_2\otimes e_{j+1}^*-e_2\wedge e_j\otimes e_1\otimes e_{j+1}^*
+\frac{l-3}2e_1\wedge e_2\otimes e_j\otimes e_{j+1}^*=0.
\end{multline*}
The only two summands that are not annihilated by $x_2$ are the summand with $i=3$ and the last summand (outside the 
summation sign). $x_2$ brings them to
$$
e_1\wedge e_2\otimes e_2\otimes e_3^*+\frac{l-3}2e_1\wedge e_2\otimes e_2\wedge e_3^*
$$
and
$$
-\frac {l-1}2 e_1\wedge e_2\otimes e_2\otimes e_3^*,
$$
respectively. The sum of these two expressions is zero. Finally, 
if we apply
$x_1$ to $v$, we get
$$
\sum_{i=3}^{l+1}(e_1\wedge e_i\otimes e_1\otimes e_i^*
-e_1\wedge e_i\otimes e_1\otimes e_i^*)
+\frac{l-1}2 e_1\wedge e_2\otimes (e_1\otimes e_2^* - e_1\otimes e_2^*)=0.
$$

$v$ is a vector of weight $\veps_1+\veps_2=\varpi_2$, so it is a highest weight vector 
in $V(\varpi_2)\subset V^*\otimes R(\mathfrak l)$. Denote its image under the embedding 
$R(\mathfrak l)\hookrightarrow V^*\otimes V$ by $w$. Then $w$ induces a linear map from $V^*$ to $V^*\otimes V^*$, and
$w\in S^2(V^*)\otimes V$ if and only if the image of this map is a subspace of $S^2(V^*)\subset V^*\otimes V^*$.
But the map induced by $w$ maps $e_l\wedge e_{l+1}$ to
$
e_1\wedge e_l\otimes e_2\wedge e_{l+1}+e_1\wedge e_{l+1}\otimes e_2\wedge e_l
-e_2\wedge e_l\otimes e_1\wedge e_{l+1}-e_2\wedge e_{l+1}\otimes e_1\wedge e_l
+(l-3)e_1\wedge e_2\otimes e_l\wedge e_{l+1}\notin S^2(V^*).
$
(Again, this is an element of $\Lambda^2(V^*)$ if $l=3$.)

\subsubsection{$\lambda=\varpi_p$, where $2<p\le \lceil l/2\rceil$}

In this case there exist no nontrivial multiplications, and we are going to prove this by a contradiction with Lemma \ref{weightlatticefailure}.
We have $\mathfrak X(V)=\{\veps_{k_1}+\ldots+\veps_{k_p}\mid 1\le k_1 < k_2 < \ldots < k_p\le l+1\}$ and 
$\lambda^*=\varpi_{l+1-p}=-\veps_{l+2-p}-\ldots-\veps_{l+1}$. From \cite[Table 5]{vinonisch}, we see that 
$V^*\otimes R(\mathfrak l)\cong V(\varpi_1+\varpi_{l+1-p}+\varpi_l)\oplus V(\varpi_1+\varpi_{l-p})
\oplus V(\varpi_{l+2-p}+\varpi_l)\oplus V(\varpi_{l+1-p})$. So, in Lemma \ref{weightlatticefailure}, there are four possibilities 
for $\lambda^*+\gamma$:

1. $\lambda^*+\gamma=\varpi_1+\varpi_{l+1-p}+\varpi_l$, then 
$\gamma=\varpi_1+\varpi_l=\veps_1+\veps_{l+1}$, and for $\nu=\veps_2+\ldots+\veps_p+\veps_{l+1}$
we have $\nu+\gamma=\veps_1+\ldots+\veps_p\in\mathfrak X(V)$
and $\nu+\gamma+\lambda^*=\veps_1+\ldots+\veps_p-\veps_{l+2-p}-\ldots-\veps_{l+1}\notin\Phi^+$.

2. $\lambda^*+\gamma=\varpi_{l+2-p}+\varpi_l$, then 
$\gamma=\varpi_{l+2-p}-\varpi_{l+1-p}+\varpi_l=\veps_{l+2-p}-\veps_{l+1}$. 
Set $\nu=\veps_1+\ldots+\veps_{p-1}+\veps_{l+1}$. 
We have $\nu+\gamma=\veps_1+\ldots+\veps_{p-1}+\veps_{l+2-p}$. 
We chose $p$ so that $2p\le l+1<l+3$, so $p-1<l+2-p$, and $\nu+\gamma\in\mathfrak X(V)$.
On the other hand, $\nu+\gamma+\lambda^*=
\veps_1+\ldots+\veps_{p-1}+\veps_{l+2-p}-\veps_{l+2-p}-\ldots-\veps_{l+1}=\veps_1+\ldots+\veps_{p-1}
-\veps_{l+3-p}-\ldots-\veps_{l+1}\notin \Phi$ since $p>2$.

3. $\lambda^*+\gamma=\varpi_1+\varpi_{l-p}$, $\gamma=\varpi_1+\varpi_{l-p}-\varpi_{l+1-p}=\veps_1-\veps_{l+1-p}$.
Take $\nu=\veps_2+\ldots+\veps_{p-1}+\veps_{l+1-p}+\veps_{l+1}\in\mathfrak X(V)$ since 
$2p-1<l+1$. Then $\nu+\gamma=\veps_1+\ldots+\veps_{p-1}+\veps_{l+1}\in\mathfrak X(V)$ and
$\nu+\gamma+\lambda^*=\varpi_1+\ldots+\veps_{p-1}
-\veps_{l+2-p}-\ldots-\veps_l\notin\Phi$ since $p>2$.

4. $\lambda^*+\gamma=\varpi_{l+1-p}$ is the highest weight of an irreducible subrepresentation 
of $V^*\otimes R(\mathfrak l)$, but in this case $\gamma=0\notin\Phi^+$, so this $\gamma$ cannot be 
the a weight $\gamma$ whose existence is guaranteed by Lemma \ref{weightlatticefailure}.

So, in each case we have a contradiction with Lemma \ref{weightlatticefailure}, therefore in this case 
nontrivial $\mathfrak l$-compatible multiplications do not exist.

\subsection{Algebras $\mathfrak l$ of type $B_l$ ($l\ge 2$) and $D_l$ ($l\ge 4$)}

The dual root system to a root system of type $B_l$ is a root system of type $C_l$, and the highest 
short root of $C_l$ is $\alpha_1+2(\alpha_2+\ldots+\alpha_{l-1})+\alpha_l$, see \cite[Section 12.2, Table 2]{humps}.
Proposition \ref{shortroot} implies that it is sufficient to consider the highest weight representations 
with highest weights $\varpi_1$ and $\varpi_l$. Lie algebras of types $B_l$ and $C_l$ are isomorphic, 
and this isomorphism identifies the highest weight representations of a Lie algebra of type $B_l$ 
with highest weights $\varpi_1$ and $\varpi_2$ with 
the highest weight representations of a Lie algebra of type $C_l$ 
with highest weights $\varpi_2$ and $\varpi_1$, respectively. We are going to consider the representation 
of a type $B_l$ Lie algebra with highest weight $\varpi_2$ later as the 
representation 
of a type $C_l$ Lie algebra with highest weight $\varpi_1$.

$D_l$ is a self-dual root system. 
All its roots have the same length, and the highest root is 
$\alpha_1+2(\alpha_2+\ldots+\alpha_{l-2})+
\alpha_{l-1}+\alpha_l$, see \cite[Section 12.2, Table 2]{humps}. By Proposition \ref{shortroot}, 
we have to consider weights $\varphi_1$, $\varpi_{l-1}$, and $\varpi_l$.
There exists a diagram automorphism that interchanges simple roots $\alpha_{l-1}$ and $\alpha_l$, so 
it interchanges representations $V(\varpi_{l-1})$ and $V(\varpi_l)$. Hence, it suffices to consider 
only one of the representations $V(\varpi_{l-1})$ and $V(\varpi_l)$. We will consider $V(\varpi_l)$.

To deal with the irreducible representations with highest weight $\varpi_l$ for both algebra types $B_l$ and $D_l$, 
we need an exact construction for these root systems. Let $\veps_1,\ldots,\veps_l$ be the orthonormal basis 
of an $l$-dimensional Euclidean space. By coordinates of vectors from this space we understand their coordinates 
with respect to this basis, unless stated otherwise. 
All vectors 
of the form $\veps_i+\veps_j, \veps_i-\veps_j, -\veps_i-\veps_j$ ($i\ne j$)
form a root system of type $D_l$. The vectors 
$\veps_i+\veps_j, \veps_i-\veps_j$ ($i<j$) form a positive root subsystem. To construct a root system 
of type $B_l$, take all vectors we took for $D_l$ and all vectors $\pm\veps_i$. The vectors $\veps_i$ 
together with the positive root subsystem of $D_l$ we chose form a positive root subsystem of $B_l$.
For both root systems, we have $\varpi_1=\veps_1$, $\varpi_2=\veps_1+\veps_2$, and $\varpi_l=(\veps_1+\ldots+\veps_l)/2$.
For both algebra types, $\varpi_2$ is the highest weight of the adjoint representation.

\subsubsection{$\lambda=\varpi_1$, $l\ge 2$ if $\mathfrak l$ is of type $B_l$}

Let $V$ be a vector space of dimension $2l$ (if $\mathfrak l$ is of type $D_l$) or $2l+1$
(if $\mathfrak l$ is of type $B_l$), and let $\omega$ be a nonsingular bilinear form on $V$. 
Then $\mathfrak l$ acts on $V$ by skew-symmetric operators, and $\rho(\mathfrak l)$ consists
of all operators skew-symmetric with respect to $\omega$.

We 
prove that nontrivial $\mathfrak l$-compatible multiplications do not exist.
Assume that we have an $\mathfrak l$-compatible multiplication on $V$.
Define a trilinear form $c$ by $c(u,v,w)=\omega (uv,w)$ ($u,v,w\in V$).
For each $u\in V$, $\mu_u$ is a skew-symmetric operator. We have
$c(u,v,w)=\omega(uv, w)=\omega (\mu_u v, w)=-\omega(v,\mu_u w)=-\omega(v,uw)=-\omega (uw,v)=-c(u,w,v)$.
Since the multiplication is commutative, we have $c(u,v,w)=\omega(uv,w)=\omega(vu,w)=c(v,u,w)$. Therefore,
$c(u,v,w)=-c(u,w,v)=-c(w,u,v)=c(w,v,u)=c(v,w,u)=-c(v,u,w)=-c(u,v,w)$, so $c(u,v,w)=0$, and $c=0$.

\subsubsection{Algebra $\mathfrak l$ of type $B_l$, $\lambda=\varpi_l$, $l\ge 3$}

There are no nontrivial $\mathfrak l$-compatible multiplications. We use Lemma \ref{weightlatticefailure}.
The Weyl group is generated by all permutations of the basis vectors and by all reflections that map $\veps_i$ to 
$-\veps_i$ and keep all other basis vectors unchanged. The orbit of $\varpi_l$ under the action of these 
group consists of all vectors such that all their coordinates equal $\pm1/2$. Since $\dim V(\varpi_l)=2^l$ (see \cite[Table 5]{vinonisch}),
these weights are all weights of $V(\varpi_l)$. The Dynkin diagram of type $B_l$ has no nontrivial 
automorphisms, so $\lambda^*=\lambda$, and we see from \cite[Table 5]{vinonisch} that 
$V^*\otimes R(\mathfrak l)\cong V(\varpi_2+\varpi_l)\oplus V(\varpi_1+\varpi_l)
\oplus V(\varpi_l)$. We have three possibilities for $\lambda^*+\gamma$.

1. $\lambda^*+\gamma=\varpi_2+\varpi_l$, $\gamma=\varpi_2=\veps_1+\veps_2$.
Set $\nu=(-\veps_1-\veps_2+\veps_3+\ldots+\veps_l)/2$, then 
$\nu+\gamma=(\veps_1+\ldots+\veps_l)/2\in\mathfrak X(V)$ and 
$\nu+\gamma+\lambda^*=
\veps_1+\ldots+\veps_l\notin\Phi^+$.

2. $\lambda^*+\gamma=\varpi_1+\varpi_l$, $\gamma=\varpi_1=\veps_1$.
In this case set $\nu=(-\veps_1+\veps_2+\ldots+\veps_l)/2$,
then $\nu+\gamma=(\veps_1+\ldots+\veps_l)/2\in\mathfrak X(V)$
and $\nu+\gamma+\lambda^*=
\veps_1+\ldots+\veps_l\notin\Phi^+$.

3. If $\lambda^*+\gamma=\varpi_l$, then $\gamma=0\notin\Phi^+$.

\subsubsection{Algebra $\mathfrak l$ of type $D_l$, $\lambda=\varpi_l$, $l\ge 4$}

Again there are no nontrivial $\mathfrak l$-compatible multiplications, and again we use Lemma \ref{weightlatticefailure} to prove this.
If $l=4$, then there exists a diagram automorphism 
of $\mathfrak l$ that 
interchanges
$V(\varpi_l)$ 
and 
$V(\varpi_1)$, 
and the case $\lambda=\varpi_1$ was considered earlier, so we may suppose that $l\ge 5$.
This time the Weyl group is generated by all permutations of the basis vectors and by all reflections that map $\veps_i$ to 
$-\veps_i$, $\veps_j$ to $-\veps_j$ and keep all other basis vectors unchanged. The orbit of $\varpi_l$ 
consists of all vectors such that all their coordinates equal $\pm 1/2$, and the number of coordinates equal to $-1/2$ is even.
This time $\dim V(\varpi_l)=2^{l-1}$, so again these vectors are all weights of $V(\varpi_l)$. In particular, the lowest weight 
is $(-\veps_1-\ldots-\veps_l)/2=-\varpi_l$ if $l$ is even, and is $(-\veps_1-\ldots-\veps_{l-1}+\veps_l)/2=-\varpi_{l-1}$ if $l$ is odd.
Hence, $\lambda^*=\varpi_l$ if $l$ is even, and $\lambda^*=\varpi_{l-1}$ if $l$ is odd. Denote $\zeta=\varpi_l$ if 
$l$ is odd, and $\zeta=\varpi_{l-1}$ is $l$ is even. In other words, $\lambda^*\ne \zeta$ and $\{\lambda^*,\zeta\}=\{\varpi_{l-1},\varpi_l\}$.
Using \cite[Table 5]{vinonisch}, we find that $V^*\otimes R(\mathfrak l)\cong V(\varpi_2+\lambda^*)\oplus V(\varpi_1+\zeta)
\oplus V(\lambda^*)$. We have to consider three cases.

1. $\lambda^*+\gamma=\varpi_2+\lambda^*$, $\gamma=\varpi_2$. Set 
$\nu=(-\veps_1-\veps_2+\veps_3+\ldots+\veps_l)/2$, then 
$\nu+\gamma=(\veps_1+\ldots+\veps_l)/2\in\mathfrak X(V)$. 
If $l$ is even, 
$\nu+\gamma+\lambda^*=
\veps_1+\ldots+\veps_l$, and 
if $l$ is odd, 
$\nu+\gamma+\lambda^*=
\veps_1+\ldots+\veps_{l-1}$. In both cases, this is not a positive root.


2.  $\lambda^*+\gamma=\varpi_1+\zeta$, $\gamma=\varpi_1+\zeta-\lambda^*$. Observe that $\zeta-\lambda^*=-(-1)^l\veps_l$, so $\gamma=\varpi_1-(-1)^l\veps_l$.
Set $\nu=(-\veps_1-(-1)^l\veps_2+\veps_3+\ldots+\veps_{l-1}+(-1)^l\veps_l)/2$, then 
$\nu+\gamma=
(\veps_1-(-1)^l\veps_2+\veps_3+\ldots+\veps_{l-1}-(-1)^l\veps_l)/2\in\mathfrak X(V)$.
If $l$ is even, then $\nu+\gamma+\lambda^*=\veps_1+\veps_3+\veps_4+\ldots+\veps_{l-1}\notin\Phi$ since $l>4$.
If $l$ is odd, then $\nu+\gamma+\lambda^*=\veps_1+\veps_2+\ldots+\veps_{l-1}\notin\Phi$.

3. If $\lambda^*+\gamma=\varpi_l$, then $\gamma=0\notin\Phi^+$.

\subsection{Algebra $\mathfrak l$ of type $C_l$ ($l\ge 2$)}

Let $E$ be a Euclidean space with an orthogonal basis $\veps_1,\ldots,\veps_l$.
Vectors of the form $\veps_i+\veps_j$, $-\veps_i-\veps_j$, $\veps_i-\veps_j$ ($i\ne j$), and
$\pm 2\veps_i$ form a root system of type $C_l$. Vectors
$\veps_i-\veps_j$ ($1\le i < j\le l$) and $\veps_i+\veps_j$ ($1\le i,j\le l$) form a system
of positive roots $\Phi^+$. 
The corresponding fundamental weights can be written as $\varpi_i=\veps_1+\ldots+\veps_i$.
The dual root system is $B_l$, and its highest short root equals
$\alpha_1^{B_l}+\alpha_2^{B_l}+\ldots+\alpha_l^{B_l}$, where $\alpha_i^{B_l}$ are simple roots of $B_l$.
So, we have to consider all fundamental representations of $\mathfrak l$. If $l=2$, then we have 
already considered the case of the first fundamental representation of an algebra of type $B_2$, 
which is isomorphic to the case second fundamental representation of an algebra of type $C_2$, so 
we do not have to consider this case again.

\subsubsection{$\lambda=\varpi_1$, $l\ge 2$}\label{spnontrivmult}
In this case, nontrivial $\mathfrak l$-compatible multiplications exist, and we will construct them.

Let $V$ be a vector space of dimension $2l$. Choose a basis $e_1,\ldots,e_l,e_{-1},\ldots,e_{-l}$ of $V$, and
let $\omega$ be the skew-symmetric bilinear form defined by $\omega=\sum(e_i^*\otimes e_{-i}^*-e_{-i}^*\otimes e_i^*)$.
We can identify $\mathfrak l$ with the Lie algebra $\mathfrak{sp}(V)$ of all operators $V\to V$ that preserve $\omega$.
Then $V$ becomes the first fundamental representation of $\mathfrak l$. All upper-triangular matrices in 
$\mathfrak{sp}(V)$ form a maximal solvable subalgebra with 
a Cartan subalgebra formed by 
all diagonal matrices in $\mathfrak{sp}(V)$ and the maximal unipotent subalgebra formed by all upper-unitriangular 
matrices in $\mathfrak{sp}(V)$. After these identifications, Chevalley generators can be written as 
$x_i=e_i\otimes e_{i+1}^*-e_{-(i+1)}\otimes e_{-i}^*$, $y_i=e_{i+1}\otimes e_i^*-e_{-i}\otimes e_{-(i+1)}^*$,
$h_i=e_i\otimes e_i^*-e_{i+1}\otimes e_{i+1}^*-e_{-i}\otimes e_{-i}^*+e_{-(i+1)}\otimes e_{-(i+1)}^*$. In particular, all $x_i$ generate 
the unipotent subalgebra of all upper-unitriangular matrices in $\mathfrak{sp}(V)$.

The form $\omega$ identifies $V$ and $V^*$ ($v\in V$ is identified with $\omega(v,\cdot)\in V^*$). Hence, we can identify 
$V^*\otimes V^*\otimes V$ with $V^*\otimes V^*\otimes V^*$, and structure tensor of (non necessarily $\mathfrak l$-compatible)
multiplications on $V$ are in one-to-one correspondence with trilinear forms on $V$. Given a multiplication, the corresponding 
form $c$ is defined as follows: $c(u,v,w)=\omega(uv,w)$. Let us reformulate the definition of $\mathfrak l$-compatibility in terms of 
the corresponding trilinear form.

Commutativity is equivalent to the equality $\omega(uv,w)=\omega (vu,w)$ for all $u,v,w\in V$ since $\omega$ is nondegenerate.
In terms of $c$ this means that $c(u,v,w)=c(v,u,w)$. An operator $\mu_u$ acts as an element of $\mathfrak l$ if and only if 
it is skew-symmetric with respect to $\omega$, i.~e. $\omega(\mu_u v,w)=-\omega(v,\mu_u w)$. In other words (using the skew symmetry of $\omega$) 
we can write $\omega (uv,w)=\omega(uw,v)$. In terms of $c$ this means that $c(u,v,w)=c(u,w,v)$. Hence, a multiplication is commutative 
and all multiplication operators act as elements of $\mathfrak l$ if and only if $c$ is a totally symmetric trilinear form. So, in what 
follows we will consider only totally symmetric trilinear forms $c$.

If a multiplication is commutative and associative 
and all multiplication operators
are skew-symmetric, then for all $u,v,w,z\in V$ we have $\omega (uvw,z)=-\omega (vw, uz)=-\omega (wv, uz)=\omega (v,wuz)=\omega (v,(wu)z)=-\omega ((wu)v,z)-\omega (uvw,z)$, 
therefore $uvw=0$ for every triple $u,v,w\in V$. And vice versa, if a multiplication is commutative and every product of the form $u(vw)$ equals zero, 
then every product of the form $(uv)w$ equals $w(uv)=0$, so the multiplication is also associative.
We also see that the nilpotency of multiplication operators in case of this representation follows
from the other three conditions in the definition of $\mathfrak l$-compatibility.

Now suppose that we deal with an $\mathfrak l$-compatible multiplication. 
Let $X$ be the linear span of all products of the form $uv$ ($u,v\in V$).
As we already know, $uvw=0$ for all $u,v,w\in V$, so for all $u,v,w,z\in V$ we have $\omega (uv,wz)=-\omega (uvw,z)=0$, 
hence $X$ is an isotropic subspace. Denote the $\omega$-orthogonal complement of $X$ by $Y$.
If $u\in Y$ and $v,w\in V$, then $vw\in X$, and $\omega (uv,w)=\omega(vu,w)=-\omega (u,vw)=0$, 
hence $uv=0$ for all $u\in Y$, $v\in V$. In other words, $\mu_u=0$ if $u\in Y$.
We have the following condition for the form $c$: if $u\in Y$, $v,w\in V$, then 
$c(u,v,w)=\omega(uv,w)=0$, so $c(Y,V,V)=0$.

Now suppose that $c$ is a totally symmetric trilinear form on $V$, and $Y'$ is a coisotropic subspace of $V$ 
such that $c(Y',V,V)=0$. (We do not assume now a priori that the corresponding multiplication is 
associative, but we already know that since $c$ is symmetric, the corresponding multiplication is commutative and 
the multiplication operators are skew-symmetric.) Then for all $u\in Y'$, $v,w\in V$ we can write $\omega (uv,w)=c(u,v,w)=0$, so $\mu_u=0$.
Denote by $X'$ the $\omega$-orthogonal complement of $Y'$.
Since $c$ is symmetric, we can also write $c(V,V,Y')=0$, so if $u,v\in V$ and $w\in Y'$, then 
$\omega (uv,w)=c(u,v,w)=0$, so $\omega (uv,Y')=0$, and $uv\in X'$. $Y'$ is coisotropic, so $\omega (X',X')=0$.
Now, using the skew symmetry of all multiplication operators, we see that for all $u,v,w,z\in V$, one has
$\omega(u(vw),z)=-\omega(vw,uz)=0$, so $u(vw)=0$, and the multiplication is associative.

Therefore, $\mathfrak l$-compatible multiplications on $V$ are in bijection with trilinear symmetric forms 
$c$ on $V$ such that there exists a coisotropic subspace $Y\subseteq V$ such that $c(Y,V,V)=0$.

The action of $Sp_{2l}$ on $V$ can move any coisotropic subspace $Y$ to an isotropic subspace containing $Z=\langle e_1,\ldots, e_l\rangle$, 
and if a subspace of $V$ contains $Z$, it is always coisotropic.
Therefore, we have a bijection between the $\mathfrak l$-compatible multiplications 
up to the action of $Sp_{2l}$ and the trilinear forms $c$ on $V$ 
such that $c(Z,V,V)=0$ up to the action of the (maximal) subgroup of $Sp_{2l}$ that preserves $Z$. This 
subgroup is exactly the parabolic subgroup of $Sp_{2l}$ that we have previously denoted by $P_l$.
The symmetric trilinear forms on $V$ such that $c(Z,V,V)=0$ are canonically identified with 
symmetric trilinear forms on $V/Z$. If $g$ is an element of the unipotent radical of 
$P_l$,
then 
the action of $g$ on $V/Z$ is 
trivial, so for every $u,v,w\in V$ 
we have $(gc)(u,v,w)=c(g^{-1}u, g^{-1}v, g^{-1}w)=c(u+u',v+v',w+w')$, where $u',v',w'\in Z$, and 
$(gc)(u,v,w)=c(u,v,w)$. The quotient of $P_l$ modulo its unipotent radical equals $GL(V/Z)$, and finally 
we get the following parametrization: the $\mathfrak l$-compatible multiplications on $V$ up to the action of 
$Sp_{2l}$ are parametrized by the trilinear symmetric forms on $V/Z$ up to the action of $GL(V/Z)$.
This action enables one to multiply any trilinear form, and therefore the structure constant tensor of the 
corresponding multiplication, by any scalar, so the action of a central torus of a reductive group with one of its simple components 
of type $C_l$ does not change the answer.

\subsubsection{$\lambda=\varpi_p$, $l\ge 3$, $p\ge 2$}

In this case there are no nontrivial $\mathfrak l$-compatible multiplications, and we use Lemma \ref{weightlatticefailure}
to prove this.
The Dynkin diagram of type $C_l$ has no nontrivial 
automorphisms, hence $\lambda^*=\lambda$. The Weyl group 
is generated by permutations of vectors $\veps_i$ and by reflections that map $\veps_i$ to $-\veps_i$ and 
keep all other basis vectors unchanged. Hence, $\mathfrak X(V)$ at least contains 
all possible linear combinations of $p$ of the basis vectors $\veps_i$ with coefficients $\pm1$. 
Using \cite[Table 5]{vinonisch}, we find that 
$V^*\otimes R(\mathfrak l)\cong V(2\varpi_1+\varpi_p)\oplus V(\varpi_1+\varpi_{p-1})
\oplus V(\varpi_p)\oplus V(\varpi_1+\varpi_{p+1})$, and the last summand is present only if $p<l$.
All irreducible components are different, and we have to consider four cases.

1. $\lambda^*+\gamma=2\varpi_1+\varpi_p$, $\gamma=2\varpi_1=2\veps_1$.
Set $\nu=-\veps_1+\veps_2+\ldots+\veps_p$, then $\nu+\gamma=
\veps_1+\ldots+\veps_p\in\mathfrak X(V)$ and $\nu+\gamma+\lambda^*=
2(\veps_1+\ldots+\veps_p)\notin\Phi^+$.

2. $\lambda^*+\gamma=\varpi_1+\varpi_{p-1}$, $\gamma=\varpi_1+\varpi_{p-1}-\varpi_p
=\veps_1-\veps_p$. If $p>2$, we can take $\nu=-\veps_1+\veps_2+\ldots+\veps_p$, then 
$\nu+\gamma=\veps_2+\ldots+\veps_{p-1}\in\mathfrak X(V)$ and 
$\nu+\gamma+\lambda^*=\veps_1+2(\veps_2+\ldots+\veps_{p-1})+\veps_p\notin\Phi^+$.
If $p=2$, we take $\nu=\veps_2+\veps_3$ (recall that $l\ge 3$), then
$\nu+\gamma=\veps_1+\veps_3\in\mathfrak X(V)$ and $\nu+\gamma+\lambda^*=2\veps_1+
\veps_2+\veps_3\notin\Phi^+$.

3. $\lambda^*+\gamma=\varpi_1+\varpi_{p+1}$. This case is only possible if $p<l$. Then 
$\gamma=\varpi_1+\varpi_{p+1}-\varpi_p=\veps_1+\veps_{p+1}$. Set 
$\nu=-\veps_1+\veps_2+\ldots+\veps_p$, then $\nu+\gamma=\veps_2+\ldots+\veps_{p+1}\in\mathfrak X(V)$
and $\nu+\gamma+\lambda^*=\veps_1+2(\veps_2+\ldots+\veps_p)+\veps_{p+1}\notin\Phi^+$.

4. If $\lambda^*+\gamma=\varpi_p$, then $\gamma=0\notin\Phi^+$.

\subsection{Algebra $\mathfrak l$ of type $E_6$}

$E_6$ is a self-dual root system, and the highest root equals $\alpha_1+2\alpha_2+2\alpha_3+3\alpha_4+2\alpha_5+\alpha_6$
(see \cite[Section 12.2, Table 2]{humps}), 
where $\alpha_1,\ldots,\alpha_6$ are simple roots. There exists an outer automorphism of $\mathfrak l$ that 
interchanges $V(\varpi_1)$ and $V(\varpi_6)$, so it is sufficient to consider the case $V=V(\varpi_1)$.

To construct the root system and the weight system for $\mathfrak l$ we use a model associated with a grading as described in \cite[Chapter 5, \S2]{vol41}.
The extended simple root system for $\mathfrak l$
consists of the simple roots $\alpha_1,\ldots,\alpha_6$ and the lowest (negative) root $\alpha'=-\alpha_1-2\alpha_2-2\alpha_3-3\alpha_4-2\alpha_5-\alpha_6$,
which is orthogonal to all simple roots except $\alpha_2$.
Consider the grading on $\mathfrak l$ corresponding to an inner automorphism and defined by label 1 at $\alpha_2$ and by labels 0 at all other 
simple roots and at the lowest root (see \cite[Chapter 3, \S3.7]{vol41}). The zeroth graded component is isomorphic to $\mathfrak{sl}_6\oplus \mathfrak{sl}_2$
as a lie algebra. The construction of a grading also provides a system of simple roots for the zeroth graded component, in this case 
the simple roots of $\mathfrak{sl}_6$ are $\alpha_1,\alpha_3,\alpha_4,\alpha_5,\alpha_6$, and the simple root of $\mathfrak{sl}_2$ is $\alpha'$.
The first graded component is an irreducible representation of the zeroth graded component, and its lowest weight is $\alpha_2$.
Hence, the first graded component is isomorphic to $V_{\mathfrak{sl}_6}(\varpi_3)\otimes V_{\mathfrak{sl}_2}(\varpi_1)=
\Lambda^3(\mathbb C^6)\otimes \mathbb C^2$ as a representation of 
$\mathfrak{sl}_6\oplus \mathfrak{sl}_2$.

These data enable us to construct a root system of type $E_6$. Consider a Euclidean space $E$ with an 
orthogonal basis $\widetilde{\veps_1}, \ldots, \widetilde{\veps_6}$ and its subspace $E'=\langle\widetilde{\veps_1}+\ldots+\widetilde{\veps_6}\rangle^\bot$.
Denote the orthogonal projection $E\to E'$ by $q$. Denote $\veps_i=q(\widetilde{\veps_i}/\sqrt2)$, then $\veps_1+\ldots+\veps_6=0$, $(\veps_i,\veps_i)=5/12$, and 
$(\veps_i,\veps_j)=-1/12$ if $i\ne j$. Consider also a one-dimensional Euclidean space with an orthonormal basis $\zeta_1$.
Then a root system of type $E_6$ consists of the following vectors in $E'\oplus \langle\zeta_1\rangle$: 
$\veps_i-\veps_j$ ($i\ne j$), $\pm \zeta_1$, and 
$\veps_i+\veps_j+\veps_k\pm \zeta_1/2$ 
($i$, $j$, and $k$ are three different indices).
One checks easily that the length of each of these vectors is 1.
Here $\alpha_1=\veps_1-\veps_2$, $\alpha_i=\veps_{i-1}-\veps_i$ for $i=3,4,5,6$, $\alpha'=\zeta_1$, and 
$\alpha_2=\veps_4+\veps_5+\veps_6-\zeta_1/2$
(the lowest weight of $\Lambda^3(\mathbb C^6)\otimes \mathbb C^2$).


To describe $\mathfrak X(V_{\mathfrak l}(\varpi_1))$, consider a Lie algebra $\mathfrak g$ of type $E_7$. 
Fix a Cartan subalgebra and the corresponding root system of $\mathfrak g$. This root system contains a root system 
of type $E_6$, so $\mathfrak l$ can be embedded into $\mathfrak g$ so that the chosen Cartan subalgebras, the chosen Borel subalgebras and 
the corresponding root systems are also 
embedded. Then simple roots are mapped to simple roots. Without loss of generality we may assume that this embedding preserves scalar multiplication.
So we may use the same notation for the simple roots of $E_7$ and for the simple roots of $E_6$, i.~e. 
we may denote the simple roots of $E_7$ by $\alpha_i$ ($1\le i\le 7$), 
then the simple root system of $E_6$ chosen previously is exactly $\{\alpha_1,\ldots,\alpha_6\}$.
Denote the parabolic subalgebra of $\mathfrak g$ corresponding to $\alpha_7$ by $\mathfrak p$.
The semisimple
part of the standard Levi subalgebra of 
$\mathfrak p$ is 
exactly $\mathfrak l$.
Denote the unipotent radical of $\mathfrak p$ by $\mathfrak u$.
The highest root of $E_7$ equals 
$\alpha=2\alpha_1+2\alpha_2+3\alpha_3+4\alpha_4+3\alpha_5+2\alpha_6+\alpha_7$.
The scalar product of $\alpha$ and any simple root of $E_7$ except $\alpha_1$
equals 0, and $(\alpha,\alpha_1)=1/2$ (recall that we initially chose the root system of type $E_6$ so that all roots are of length 1).
There are 36 positive roots in $E_6$ and 63 positive roots in $E_7$, so $\dim\mathfrak u=27$. 
But $\dim V=27$ (see \cite[Table 5]{vinonisch}), so $\mathfrak u$ is
isomorphic to $\dim V$ as a representation of $\mathfrak l$.
Therefore, the numerical label of a weight of $V$
at a root of $\mathfrak l$ can be computed as twice the scalar product of the root of $E_7$ 
corresponding to this weight (its decomposition into a linear combination of simple roots of $E_7$
contains $\alpha_7$ with coefficient 1) and the root of $E_6$ considered as a root of $E_7$.

The embedding $\mathfrak{sl}_6\oplus\mathfrak{sl}_2\subset \mathfrak l$ enables us to consider $\mathfrak u$ as a representation 
of $\mathfrak{sl}_6\oplus\mathfrak{sl}_2$. Let us decompose it into a sum of irreducible 
$\mathfrak{sl}_6\oplus\mathfrak{sl}_2$-representations.
$E_7$ has a root subsystem of type $A_6$ generated  by $\alpha_1, \alpha_3,\alpha_4,\alpha_5,\alpha_6,\alpha_7$. Hence, 
$\beta_1=\alpha_1+\alpha_3+\alpha_4+\alpha_5+\alpha_6+\alpha_7$ is a root of $E_7$. $(\alpha_i,\beta_1)\ge 0$ if $i\ne 2$, so $(\alpha_i+\beta_1,\alpha_i+\beta_1)>1$ if 
$i\ne 2$, and this is not a root of $E_7$. Also, if we add $\alpha'$ to $\beta_1$, we get a linear combination of roots $\alpha_i$ where 
some coefficients are positive and some are negative, so this sum cannot be a root of $E_7$. Therefore, $\mathfrak g_{\beta_1}$ is the highest weight 
subspace of an irreducible $\mathfrak{sl}_6\oplus\mathfrak{sl}_2$-subrepresentation in $\mathfrak u$. The only nonzero numerical label of $\beta_1$ at the 
chosen simple roots for $\mathfrak{sl}_6$ is the one at $\alpha_1$, and this numerical label equals 1. The numerical label of $\beta_1$ at 
$\alpha'$ also equals 1, hence this irreducible representation is isomorphic to $\mathbb C^6\otimes \mathbb C^2$.

$E_7$ also has a root subsystem of type $A_4$ generated by $\alpha_1, \alpha_3, \alpha_4, \alpha_2$, so $\alpha''=\alpha_1+\alpha_2+\alpha_3+\alpha_4$
is a root of $E_7$.
The reflection defined by $\alpha''$ maps $\alpha$ to $\alpha-\alpha''=\alpha_1+\alpha_2+2\alpha_3+3\alpha_4+3\alpha_5+2\alpha_6+\alpha_7$, 
so this is also a root of $E_7$. Denote it by $\beta_2$. Again, if $i\ne 2$, then $(\alpha_i,\beta_2)>0$ and $(\alpha_i+\beta_2,\alpha_i+\beta_2)>1$, so 
$\alpha_i+\beta_2$ cannot be a root of $E_7$. And again $\alpha'+\beta_2$ is a linear combinations of roots $\alpha_i$ where some 
coefficients are positive and some are negative, so $\alpha'+\beta_2$ is not a root of $E_7$. Again we conclude that 
$\mathfrak g_{\beta_2}$ 
is the highest weight 
subspace of an irreducible $\mathfrak{sl}_6\oplus\mathfrak{sl}_2$-subrepresentation in $\mathfrak u$.
The only nonzero numerical label of $\beta_2$ at the 
chosen simple roots for $\mathfrak{sl}_6$ is the one at $\alpha_5$, and this numerical label equals 1.
The numerical label of $\beta_2$ at $\alpha'$ equals zero, so this irreducible subrepresentation is 
isomorphic to $\Lambda^2(\mathbb C^6)^*$, and $\mathfrak{sl}_2$ acts trivially on it. Therefore, 
weights of $V$ are $-\veps_i-\veps_j$ ($1\le i< j\le 6$) and $\veps_i\pm\zeta_1/2$ ($1\le i\le 6$).

The highest weight of $V$ as an $\mathfrak l$-module is $\lambda=\varpi_1=\veps_1-\zeta_1/2$. The lowest weight is 
$\veps_6+\zeta_1/2$ since for none of the simple roots $\alpha_i$, $\veps_6+\zeta_1/2-\alpha_i$ is a weight of $V$.
So, $\lambda^*=-\veps_6-\zeta_1/2$. By computing scalar products, one checks directly that $\lambda^*=\varpi_6$.

Now we are ready to apply Lemma \ref{weightlatticefailure}. From \cite[Table 5]{vinonisch} we see that 
$V^*\otimes R(\mathfrak l)\cong V(\varpi_2+\varpi_6)+V(\varpi_3)+V(\varpi_6)$. We have to consider three cases.

1. $\lambda^*+\gamma=\varpi_2+\varpi_6$, $\gamma=\varpi_2$. One can check directly that $\varpi_2=-\zeta_1$.
Take $\nu=\veps_1+\zeta_1/2$, then $\nu+\gamma=\veps_1-\zeta_1/2\in\mathfrak X(V)$ and 
$\nu+\gamma+\lambda^*=\veps_1-\veps_6-\zeta_1\notin\Phi$.

2. $\lambda^*+\gamma=\varpi_3$, and again one can check by computing scalar products that 
$\varpi_3=\veps_1+\veps_2-\zeta_1$. So, $\gamma=\veps_1+\veps_2+\veps_6-\zeta_1/2$, 
and we can set $\nu=-\veps_2-\veps_6$. Then $\nu+\gamma=\veps_1-\zeta_1/2\in\mathfrak X(V)$
and $\nu+\gamma+\lambda^*=\veps_1-\veps_6-\zeta_1\notin\Phi$.

3. If $\lambda^*+\gamma=\varpi_6$, then $\gamma=0\notin\Phi^+$.

\subsection{Algebra $\mathfrak l$ of type $E_7$}

$E_7$ is also a self-dual root system, and the highest root equals 
$2\alpha_1+2\alpha_2+3\alpha_3+4\alpha_4+3\alpha_5+2\alpha_6+\alpha_7$. We have to consider $V=V(\varpi_7)$.
Without loss of generality, suppose that the length of each root is 1.
Again, we use a model associated with a grading to describe a root system of type $E_7$. Denote the lowest root 
of $E_7$ by $\alpha'$. We have $(\alpha',\alpha_i)=0$ for $i\ne 1$, $(\alpha',\alpha_1)=-1/2$.
Consider the grading on $\mathfrak l$ corresponding to 
an inner automorphism and defined by label 1 at $\alpha_2$ and by labels 0 
at $\alpha'$ and at all $\alpha_i$, where $i\ne 2$. The zeroth graded component 
is isomorphic to $\mathfrak{sl}_8$, and its simple roots defined by the grading construction 
are: $\alpha',\alpha_1,\alpha_3,\alpha_4,\alpha_5,\alpha_6,\alpha_7$. The first grading component 
as a representation of the zeroth graded component is isomorphic to $\Lambda^4(\mathbb C^8)$.

To construct a root system of type $E_7$, consider a Euclidean space $E$ with an orthonormal
basis $\widetilde{\veps_1},\ldots,\widetilde{\veps_8}$ and its subspace $E'=\langle\widetilde{\veps_1}+\ldots+\widetilde{\veps_8}\rangle^\bot$.
Denote the orthogonal projection $E\to E'$ by $q$. Denote $\veps_i=q(\widetilde{\veps_i}/\sqrt2)$.
One check directly that $(\veps_i,\veps_i)=7/16$, $(\veps_i,\veps_j)=-1/16$ if $i\ne j$, and 
$\veps_1+\ldots+\veps_8=0$. A root system of type $E_7$ consists of all vectors of the form 
$\veps_i-\veps_j$ ($i\ne j$) and $\veps_i+\veps_j+\veps_k+\veps_l$ (all four indices are different).
For the simple roots of $E_7$ provided by the grading construction we have
$\alpha_1=\veps_2-\veps_3$, $\alpha_i=\veps_i-\veps_{i+1}$ for $3\le i\le 7$, and 
$\alpha'=-2\alpha_1-2\alpha_2-3\alpha_3-4\alpha_4-3\alpha_5-2\alpha_6-\alpha_7=\veps_1-\veps_2$, so
$\alpha_2=\veps_5+\veps_6+\veps_7+\veps_8$. One checks easily that all these vectors are of length 1.

To find $\mathfrak X(V(\varpi_7))$, consider a Lie algebra $\mathfrak g$ of type $E_8$. 
Its root system contains a subsystem of subsystem of type $E_7$, so we can choose a Borel subalgebra and a Cartan 
subalgebra in $\mathfrak g$ and identify 
$\mathfrak l$ with a subalgebra of $\mathfrak g$ so that the chosen Borel (resp. Cartan) subalgebra of $\mathfrak l$ is 
embedded into the chosen Borel (resp. Cartan) subalgebra
of $\mathfrak g$. With this embedding, the simple roots $\alpha_1,\ldots,\alpha_7$ of $\mathfrak l$ are mapped 
to the simple roots $\alpha_1,\ldots,\alpha_7$ of $\mathfrak g$, so we can use the same notation for them 
(and denote the remaining simple root of $\mathfrak g$ by $\alpha_8$).
The highest root of $E_8$ is $2\alpha_1+3\alpha_2+4\alpha_3+6\alpha_4+5\alpha_5
+4\alpha_6+3\alpha_7+2\alpha_8$, and the reflection defined by $\alpha_8$ maps it to 
$\beta=2\alpha_1+3\alpha_2+4\alpha_3+6\alpha_4+5\alpha_5
+4\alpha_6+3\alpha_7+\alpha_8$, so $\beta$ is also a root of $E_8$. There are 
120 positive roots in $E_8$, and 63 of them are positive roots of $E_7$, so 
their decomposition into a linear combination of $\alpha_1,\ldots,\alpha_8$ does not actually contain $\alpha_8$.
We know one root in $E_8$ whose decomposition into a linear combination of simple roots contains $\alpha_8$ with 
coefficient 2, namely the highest root.
Hence, there are at most 56 positive roots in $E_8$ whose decomposition into a linear combination of simple
roots contains $\alpha_8$ with coefficient 1. Denote the direct sum of the corresponding root subspaces in $\mathfrak g$ by $W$.
Clearly, $\mathfrak l\subset \mathfrak g$ preserves $W$. A direct calculation shows that $(\alpha_i,\beta)=0$ for $1\le i\le 6$ 
and $(\alpha_7,\beta)=1/2$. So, $\alpha_i+\beta$ is not a root if $1\le i\le 7$, and 
hence $\mathfrak g_\beta$ is a highest weight subspace for the action of $\mathfrak l$ on $W$.
It also follows from the values of these scalar products that $\mathfrak g_\beta$ is a subspace of weight $\varpi_7$ in terms of the $\mathfrak l$-action.
But we know that $\dim V=56$, so $V$ is isomorphic to $W$ as an $\mathfrak l$-representation, and we can identify them.

Now let us decompose $V$ into a sum of irreducible $\mathfrak{sl}_8$-representations. $E_8$ has a subsystem 
of type $A_7$, its simple roots are $\alpha_1, \alpha_3, \alpha_4,\alpha_5,\alpha_6,\alpha_7,\alpha_8$.
Hence, $\beta_1=\alpha_1+\alpha_3+\alpha_4+\alpha_5+\alpha_6+\alpha_7+\alpha_8$ is a root of $E_8$.
It has nonnegative scalar products with $\alpha_i$ if $i=1$ or $3\le i\le 8$, hence $(\alpha_i+\beta_1,\alpha_i+\beta_1)>1$ and 
$\alpha_i+\beta_1$ is not a root of $E_8$. The decomposition of $\alpha'+\beta_1$ into a sum of simple roots contains $\alpha_8$ with coefficient 1 and
$\alpha_2$ with coefficient $-2$, so $\alpha'+\beta_1$ also is not a root of $E_8$. Therefore, $\mathfrak g_{\beta_1}$ is a highest weight 
subspace of an irreducible $\mathfrak{sl}_8$-subrepresentation of $V$. 
$\beta_1$ is orthogonal to all simple roots of $\mathfrak{sl}_8$ except $\alpha_2$, and $(\beta_1,\alpha_1)=1/2$, so 
this irreducible subrepresentation is isomorphic to $\Lambda^2(\mathbb C^8)$.

In the previous section we have seen that $\alpha_1+\alpha_2+2\alpha_3+3\alpha_4+3\alpha_5+2\alpha_6+\alpha_7$ is a root of 
$E_7$. $\alpha_6+\alpha_7+\alpha_8$ is a root of $A_7\subset E_8$. We have 
$(\alpha_1+\alpha_2+2\alpha_3+3\alpha_4+3\alpha_5+2\alpha_6+\alpha_7, \alpha_6+\alpha_7+\alpha_8)=-1/2$, so the reflection 
defined by $\alpha_6+\alpha_7+\alpha_8$ maps $\alpha_1+\alpha_2+2\alpha_3+3\alpha_4+3\alpha_5+2\alpha_6+\alpha_7$ to 
$\beta_2=\alpha_1+\alpha_2+2\alpha_3+3\alpha_4+3\alpha_5+3\alpha_6+2\alpha_7+\alpha_8$.
Now, $(\beta_2,\alpha_i)=0$ for $i=1,3,4,5,7,8$, $(\beta_2,\alpha_6)=1/2$, and $(\beta_2,\alpha')=0$.
Hence, $(\beta_2+\alpha_i,\beta_2+\alpha_i)>1$ for $i\ne 2$, $(\beta_2+\alpha',\beta_2+\alpha')>1$, and 
$\beta_2+\alpha_i$ for $i\ne 2$ and $\beta_2+\alpha'$ are not roots of $E_8$. We see that 
$\mathfrak g_{\beta_2}$ is a highest weight subspace of an irreducible $\mathfrak{sl}_8$-subrepresentation of $V$,
and that this subrepresentation is $\Lambda^2(\mathbb C^8)^*$. Finally, $\dim \Lambda^2(\mathbb C^8)=\dim \Lambda^2(\mathbb C^8)^*=28$, 
$\dim V=56$, so $\Lambda^2(\mathbb C^8)$ and $\Lambda^2(\mathbb C^8)^*$ are all irreducible $\mathfrak{sl}_8$-subrepresentations of $V$.
Therefore, $\mathfrak X(V)$ consists of all vectors of the form $\pm(\veps_i+\veps_j)$ ($1\le i < j\le 8$).

Now we prove that there are no nontrivial $\mathfrak l$-compatible multiplications on $V$ using Lemma \ref{weightlatticefailure}.
The Dynkin diagram of type $B_l$ has no nontrivial 
automorphisms, so $\lambda^*=\lambda=\varpi_7$.
From \cite[Table 5]{vinonisch} we see that 
$V^*\otimes R(\mathfrak l)\cong V(\varpi_1+\varpi_7)\oplus V(\varpi_2)
\oplus V(\varpi_7)$. A direct calculation of scalar products shows that 
$\varpi_1=\veps_2-\veps_1$, $\varpi_2=2\veps_1$, and $\varpi_7=\veps_1+\veps_8$. We have to consider three cases.

1. $\gamma+\lambda^*=\varpi_1+\varpi_7$, $\gamma=\varpi_1=\veps_2-\veps_1$.
Set $\nu=\veps_1+\veps_3$, then $\nu+\gamma=\veps_2+\veps_3\in\mathfrak X(V)$
and $\nu+\gamma+\lambda^*=2\veps_1+\veps_2+\veps_3\notin\Phi$.

2. $\gamma+\lambda^*=\varpi_2$, $\gamma=\varpi_2-\varpi_7=\veps_1-\veps_8$.
Set $\nu=\veps_7+\veps_8$, then $\nu+\gamma=\veps_1+\veps_7\in\mathfrak X(V)$ and
$\nu+\gamma+\lambda^*=2\veps_1+\veps_7+\veps_8\notin\Phi$.

3. If $\gamma+\lambda^*=\varpi_7$, then $\gamma=0\notin\Phi^+$.

\subsection{Algebra $\mathfrak l$ of type $E_8$}

$E_8$ is a self-dual root system, all roots have equal lengths, and the highest root equals 
$2\alpha_1+3\alpha_2+4\alpha_3+6\alpha_4+5\alpha_5+4\alpha_6+3\alpha_7+2\alpha_8$.
All coefficients here are grater than 1, so by Proposition \ref{shortroot}, any 
$\mathfrak l$-compatible multiplication on any $\mathfrak l$-module is trivial.

\subsection{Algebra $\mathfrak l$ of type $F_4$}

A root system $\Phi$ of type $F_4$ is isomorphic to its dual $\Phi^\vee$, but roots have different lengths.
Roots of $\Phi^\vee$ corresponding to (any) simple root system of $\Phi$ also form a system of simple roots,
but if Cartan matrices for these simple roots are the same, then the first (resp. second, third, fourth) simple root 
of $\Phi$ corresponds to the fourth (resp. third, second, first) simple root of $\Phi^\vee$. So, choose
a simple root system $\alpha_1,\alpha_2,\alpha_3,\alpha_4$ in $\Phi$, and denote the corresponding roots of $\Phi^\vee$ by
$\beta_4,\beta_3,\beta_2,\beta_1$, respectively. Then $\beta_1,\beta_2,\beta_3,\beta_4$ is a simple root system in $\Phi^\vee$, 
and the corresponding Cartan matrix is the same as for $\alpha_1,\alpha_2,\alpha_3,\alpha_4$.
In particular, the highest short root equals $\beta_1+2\beta_2+3\beta_3+2\beta_4$, 
so, since $\beta_1$ corresponds to $\alpha_4$, by Proposition \ref{shortroot} we only have to 
consider $V=V(\lambda)$, where $\lambda=\varpi_4$.

We use an explicit construction for $F_4$ (see \cite[\S12]{humps}). Consider a Euclidean space
with an orthonormal basis $\veps_1,\veps_2,\veps_3,\veps_4$. 
Then a root system of type $F_4$ is formed by all vectors 
$\pm\veps_i$, $\pm\veps_i\pm\veps_j$, ($i\ne j$), and
$(\pm\veps_1\pm\veps_2\pm\veps_3\pm\veps_4)/2$ (in all cases the signs may be chosen independently).
As a system of simple roots, we can take $\alpha_1=\veps_2-\veps_3$, $\alpha_2=\veps_3-\veps_4$, $\alpha_3=\veps_4$,
$\alpha_4=(\veps_1-\veps_2-\veps_3-\veps_4)/2$. Then a direct check shows that $\varpi_4=\veps_1$. 
Hence, $\mathfrak X(V)$ consists of all short roots 
and 0. Another direct computation shows that $\varpi_1=\veps_1+\veps_2$ and $\varpi_3=(3\veps_1+\veps_2+\veps_3+\veps_4)/2$.

We use Lemma \ref{weightlatticefailure} to prove that there are no nontrivial $\mathfrak l$-compatible multiplications.
The Dynkin diagram of type $F_4$ has no nontrivial automorphisms, so $\lambda^*=\lambda$.
From \cite[Table 5]{vinonisch} we see that $V^*\otimes R(\mathfrak l)\cong V(\varpi_1+\varpi_4)\oplus V(\varpi_3)
\oplus V(\varpi_4)$. We have to consider three cases.

1. $\gamma+\lambda^*=\varpi_1+\varpi_4$, $\gamma=\varpi_1=\veps_1+\veps_2$.
Set $\nu=-\veps_2$, then $\nu+\gamma=\veps_1\in\mathfrak X(V)$ and
$\nu+\gamma+\lambda^*=2\veps_1\notin\Phi$.

2. $\gamma+\lambda^*=\varpi_3$, $\gamma=\varpi_3-\varpi_4=(\veps_1+\veps_2+\veps_3+\veps_4)/2$.
Set $\nu=0$, then $\nu+\gamma=\gamma=(\veps_1+\veps_2+\veps_3+\veps_4)/2\in\mathfrak X(V)$, 
$\nu+\gamma+\lambda^*=\varpi_3=(3\veps_1+\veps_2+\veps_3+\veps_4)/2\notin\Phi$.

3. If $\gamma+\lambda^*=\varpi_4$, then $\gamma=0\notin\Phi^+$.

\subsection{Algebra $\mathfrak l$ of type $G_2$}

The usage of Lemma \ref{shortroot} in this case is similar to the previous case. Namely, if $\Phi$ is a root system 
of type $G_2$, then the dual root system $\Phi^+$ is also of type $G_2$, and if $\alpha_1$ and $\alpha_2$ are simple roots of $\Phi$, 
then the corresponding roots of $\Phi^\vee$ also form a simple root system, but if Cartan matrices for these simple 
root systems are the same, then the first (resp. second) simple root of $\Phi$ corresponds to the second
(resp. the first) simple root of $\Phi^\vee$. So, denote the root of $\Phi^\vee$ corresponding to $\alpha_1\in\Phi$ (resp. to $\alpha_2\in\Phi$)
by $\beta_2$ (resp. by $\beta_1$). Then the highest short root of $\Phi^\vee$ is $2\beta_1+\beta_2$, and by Lemma \ref{shortroot}
we have to consider $V=V(\lambda)$, where $\lambda=\varpi_1$.

A root system of type $G_2$ can be constructed as the union a root system of type $A_2$ and sums of two roots from $A_2$
such that the angle between them is $\pi/3$. More exactly, consider a Euclidean space $E$ with an orthonormal basis 
$\widetilde{\veps_1},\widetilde{\veps_2},\widetilde{\veps_3}$ and its subspace 
$E'=\langle\widetilde{\veps_1}+\widetilde{\veps_2}+\widetilde{\veps_3}\rangle^\bot$. Denote the 
orthogonal projection $E\to E'$ by $q$ and $\veps_i=q(\widetilde{\veps_i})$. Then all vectors 
$\veps_i-\veps_j$ ($i\ne j$) form a root system of type $A_l$. The angle between $\veps_1-\veps_2$ and $\veps_1-\veps_3$ 
equals $\pi/3$, and $\veps_1-\veps_2+\veps_1-\veps_3=3\veps_1-\veps_1-\veps_2-\veps_3=3\veps_1$ is one of the roots of $G_2$.
The remaining roots can be obtained by the action of the Weyl group of $A_2$, they equal $\pm3\veps_i$.
A system of positive roots is formed by $\veps_i-\veps_j$, where $1\le i< j\le 3$, $3\veps_1$, $3\veps_2$, and $3\veps_1+3\veps_2=-3\veps_3$.
The resulting system of simple roots consists of $\alpha_1=\veps_1-\veps_2$ and $\alpha_2=3\veps_2$. The fundamental weights 
can be written as $\varpi_1=\veps_1-\veps_3$, $\veps_2=-3\veps_3$. Hence, $\mathfrak X(V(\varpi_1))$ consists of all short roots and 0.

Again we use Lemma \ref{weightlatticefailure} to prove that there are no nontrivial $\mathfrak l$-compatible multiplications.
The Dynkin diagram of type $G_2$ has no nontrivial automorphisms, so $\lambda^*=\lambda$.
From \cite[Table 5]{vinonisch} we see that $V^*\otimes R(\mathfrak l)\cong V(\varpi_1+\varpi_2)\oplus V(2\varpi_1)
\oplus V(\varpi_1)$. There are three cases to consider.

1. $\gamma+\lambda^*=\varpi_1+\varpi_2$, $\gamma=\varpi_2=-3\veps_3$.
Take $\nu=\veps_3-\veps_2$, then $\nu+\gamma=-2\veps_3-\veps_2=\veps_1-\veps_3\in\mathfrak X(V)$
and $\nu+\gamma+\lambda^*=2\lambda^*\notin\Phi$.

2. $\gamma+\lambda^*=2\varpi_1$, $\gamma=\varpi_1=\veps_1-\veps_3$.
Take $\nu=0$, then $\nu+\gamma=\veps_1-\veps_3\in\mathfrak X(V)$ and
$\nu+\gamma+\lambda^*=2\lambda^*\notin\Phi$.

3. If $\gamma+\lambda^*=\varpi_1$, then $\gamma=0\notin\Phi^+$.

We have considered all types of simple Lie algebras, so the proof of Theorem \ref{multsclass} is now finished.

\section{Classification of generically transitive $(\Ga)^m$-actions on generalized flag varieties}
\label{finalclass}

To prove Theorem \ref{actsclass}, we apply Theorem \ref{actsandmults} to each case of a simple group $G$
and its parabolic subgroup $P$ such that $G=\Aut(G/P)^\circ$ there exists at least one generically transitive $(\Ga)^m$-action
on $G/P$ (see Introduction for the list of the cases we have to consider). We use notation from Introduction and from 
Section \ref{prelim}. To understand the action of $L$ on $\mathfrak u^-$, we argue as follows.

$L$ is always locally isomorphic to the product of the commutator subgroup of $L$, which is a semisimple group, and the center of
$L$, which is a torus. In the cases we have to consider, $P$ is a maximal parabolic subgroup, $P=P_i$. To get the Dynkin
diagram of the commutator subgroup of $L$, we remove the $i$th vertex from the Dynkin diagram of $G$. This also gives us an 
embedding of the root system of $L$ into the root system of $G$, and the subsystem of positive (resp. simple)
roots of $L$ is embedded into the system of positive (resp. simple) roots of $G$.

The $L$-action on $\mathfrak u^-$ is always faithful, the central torus of $L$ acts nontrivially on $\mathfrak u^-$.
The irreducible $L$-subrepresentations 
of $\mathfrak u^-$ are in bijection with negative roots $\beta$ such that the decomposition 
of $\beta$ into a sum of simple roots contains $\alpha_i$, and for every $j\ne i$, $\beta-\alpha_j$ is not a root from the root system of $G$.
(More precisely, $\mathfrak g_\beta$ is a lowest $L$-weight subspace of such an $L$-representation.)
In particular, if $\alpha=\sum n_j\alpha_j$ is the highest root, and $n_i=1$, then $\beta$ must be the lowest root, 
and $\mathfrak u^-$ is an irreducible $L$-representation.
In this case, it is also easy to find the highest $L$-weight subspace, namely, it equals $\mathfrak g_{-\alpha_i}$.
The numerical label of the highest $L$-weight of this representation 
at a simple root $\alpha_j$ (understood as a simple root of $L$) equals $-2(\alpha_i,\alpha_j)/(\alpha_j,\alpha_j)$.

\subsection{Group $G$ of type $A_l$, $P=P_1$ or $P=P_l$}

The subgroups $P_1$ and $P_l$ can be interchanged by a diagram automorphism, so without loss of generality 
we may suppose that $P=P_1$. Using the generic argument stated above, we conclude that $[L,L]$ is a group of 
type $A_l$, $\mathfrak u^-$ is an irreducible $L$-representation and its lowest weight is minus the last 
fundamental weight of $L$, so $\mathfrak u^-$ is a tautological 
$[L,L]$-module.
This case was considered in Section
\ref{commalgarises}. Since the central torus of $L$ is one-dimensional and acts nontrivially on $\mathfrak u^-$,
the commutative 
unipotent subalgebras $\mathfrak a\subset\mathfrak p^-$ such that 
$\mathfrak a\cap\mathfrak l=0$ and $\mathfrak a\oplus\mathfrak l=\mathfrak p^-$, considered up to $L$-conjugation,
are parametrized by isomorphism classes of $l$-dimensional 
associative commutative algebras with nilpotent multiplication operators.

To prove Theorem \ref{actsclass} in this case, we have to check that if two commutative unipotent 
subalgebras $\mathfrak a_1\subseteq \mathfrak u_0^-$ and $\mathfrak a_2\subseteq \mathfrak u_0^-$
such that $\mathfrak a_i\cap\mathfrak l=0$ and $\mathfrak a_i\oplus\mathfrak l=\mathfrak p^-$ are $P$-conjugate, 
then they are $L$-conjugate. Without loss of generality, $G=SL_{l+1}$. Take $p\in P$ such that 
$(\Ad p)\mathfrak a_1=\mathfrak a_2$. Since $P=L\ltimes U$, we can write 
$p=gu$, where $g\in L$, $u\in U$. Then $(\Ad u)\mathfrak a_1=(\Ad g^{-1})\mathfrak a_2$.
We write elements of $G=SL_{l+1}$ as block matrices with the following block sizes:
$$
\left(\begin{array}{c|c}
1\times 1 & 1\times l \\ \hline
l\times 1 & l\times l
\end{array}\right).
$$
Then $L$ is the group of all matrices of the form 
$$
\left(\begin{array}{c|c}
a & 0 \\ \hline
0 & b
\end{array}\right),
$$
where $a\in\mathbb C\setminus\{0\}$, $b\in GL_l$ and $a\det b=1$.
All elements of $\mathfrak a_1$ are matrices of the form
$$
\left(\begin{array}{c|c}
0 & 0 \\ \hline
a & b
\end{array}\right),
$$
and $a$ can be an arbitrary column vector of length $l$ since $\mathfrak a_1+\mathfrak l=\mathfrak p^-$.
All elements of $\mathfrak a_2$ are also of this form. 
$\mathfrak u$ 
$U$
is the 
group
of all matrices of the form
$$
\left(\begin{array}{c|c}
1 & a \\ \hline
0 & \id_l
\end{array}\right),
$$
where $a$ is an arbitrary row vector of length $l$.
Suppose that
$$
g=\left(\begin{array}{c|c}
x & 0 \\ \hline
0 & y
\end{array}\right),
\quad
u=\left(\begin{array}{c|c}
1 & v \\ \hline
0 & 1
\end{array}\right).
$$
Assume that $u\ne \id_{l+1}$, then $v\ne 0$.
If 
$$
a_1=\left(\begin{array}{c|c}
0 & 0 \\ \hline
a'_1 & a''_1
\end{array}\right)\in\mathfrak a_1,
$$
then
$$
u
a_1
u^{-1}=
\left(\begin{array}{c|c}
1 & v \\ \hline
0 & 1
\end{array}\right)
\left(\begin{array}{c|c}
0 & 0 \\ \hline
a'_1 & a''_1
\end{array}\right)
\left(\begin{array}{c|c}
1 & -v \\ \hline
0 & 1
\end{array}\right)=
\left(\begin{array}{c|c}
va'_1 & -va'_1v+va''_1 \\ \hline
a'_1 & -a'_1v+a''_1
\end{array}\right).
$$
Since $v\ne0$, there exists $a_1\in \mathfrak a_1$ such that the topmost leftmost entry of this matrix is nonzero.
On the other hand, if 
$$
a_2=\left(\begin{array}{c|c}
0 & 0 \\ \hline
a'_2 & a''_2
\end{array}\right)\in\mathfrak a_2,
$$
then
$$
g^{-1}a_2g=
\left(\begin{array}{c|c}
x & 0 \\ \hline
0 & y
\end{array}\right)
\left(\begin{array}{c|c}
0 & 0 \\ \hline
a'_2 & a''_2
\end{array}\right)
\left(\begin{array}{c|c}
x^{-1} & 0 \\ \hline
0 & y^{-1}
\end{array}\right)=
\left(\begin{array}{c|c}
0 & 0 \\ \hline
ya'_2x^{-1} & ya''_2y^{-1}
\end{array}\right),
$$
and the topmost leftmost entry of this matrix is always 0.
Therefore, if $u\ne \id_{l+1}$, then 
$(\Ad u)\mathfrak a_1$ and $(\Ad g^{-1})\mathfrak a_2$
cannot coincide. So, $u=\id_{l+1}$, $p=g\in L$, 
and $\mathfrak a_1$ and $\mathfrak a_2$ are $L$-conjugate.

\subsection{Group $G$ of type $A_l$, $P=P_i$, $1<i<l$}

The commutator subgroup of $G$ is locally isomorphic to $SL_i\times SL_{l+1-i}$, and 
$\mathfrak u^-$ is an irreducible $L$-representation isomorphic to $V_{SL_i}(\varpi_{i-1})\otimes
V_{SL_{l+i-1}}(\varpi_1)$. It follows from Proposition \ref{possibledecomp} that 
there are no nontrivial $\mathfrak l$-compatible multiplications in this case.

\subsection{Group $G$ is not of type $A_l$}

The proof for the remaining cases ($G$ of type $B_l$, $C_l$, $D_l$, $E_6$, or $E_7$) follows the same pattern. 
We compute the commutator subgroup of $L$, and it turns out to be a simple group. We also note that 
$\mathfrak u^-$ is an irreducible $L$-representation, and compute 
its highest $[L,L]$-weight.
Then we see from Theorem \ref{multsclass} that in this case there are no nontrivial $\mathfrak l$-compatible multiplications.
These calculations are summarized in the following table:

\medskip


\begin{center}
\begin{tabular}{c|c|c|c}
Type of $G$ & $P$ & Type of $[L,L]$  & Highest $[L,L]$-weight of $\mathfrak u^-$ \\ \hline
\hline
$B_l$ ($l\ge 3$) & $P_1$ & $B_{l-1}$ & $\varpi_1$ \\ \hline
$C_l$ ($l\ge 2$) & $P_l$ & $A_{l-1}$ & $2\varpi_{l-1}$ \\ \hline
$D_4$ & $P_1$ & $A_3$ & $\varpi_2$ \\ \hline
$D_l$ ($l\ge 5$) & $P_1$ & $D_{l-1}$ & $\varpi_1$ \\ \hline
$D_l$ ($l\ge 5$) & $P_{l-1}, P_l$ & $A_{l-1}$ & $\varpi_{l-2}$ \\ \hline
$E_6$ & $P_1,P_6$ & $D_5$ & $\varpi_5$ \\ \hline
$E_7$ & $P_7$ & $E_6$ & $\varpi_6$ 
\end{tabular}
\end{center}

This finishes the proof of Theorem \ref{actsclass}.

\end{document}